\documentclass[11pt]{article} 
\usepackage[leqno]{amsmath}
\usepackage{amssymb,amsthm,url}
\usepackage[letterpaper,margin=1in]{geometry}
\usepackage{graphicx}
\usepackage{color}
\numberwithin{equation}{section}
\numberwithin{figure}{section}
\usepackage{enumitem}		
\usepackage{multirow}
\usepackage{tablefootnote}
\usepackage{stmaryrd}
\usepackage{pict2e}

\usepackage[style=alphabetic,backend=bibtex,maxbibnames=99]{biblatex}
\definecolor{green}{RGB}{0, 0, 0}
\definecolor{blue}{RGB}{0, 0, 0}

\DeclareMathOperator{\Mat}{Mat}
\DeclareMathOperator{\Prob}{Prob}
\DeclareMathOperator{\Expect}{Expect}
\DeclareMathOperator{\id}{id}
\DeclareMathOperator{\gr}{gr}

\DeclareMathOperator{\sspan}{span}
\DeclareMathOperator{\sym}{sym}
\DeclareMathOperator{\stdbrac}{stdbrac}
\DeclareMathOperator{\des}{des}
\DeclareMathOperator{\wcomp}{wcomp}

\newcommand\calsh{\mathcal{S}}

\DeclareMathOperator{\orif}{BRiffle} 

\date{}

\theoremstyle{plain}
\newtheorem{thm}[equation]{\protect\theoremname}
  \theoremstyle{definition}
  \newtheorem{defn}[equation]{\protect\definitionname}
  \theoremstyle{plain}
  \newtheorem{prop}[equation]{\protect\propositionname}
  \theoremstyle{plain}
  \newtheorem{cor}[equation]{\protect\corollaryname}
  \theoremstyle{plain}
  \newtheorem{lem}[equation]{\protect\lemmaname}
  \theoremstyle{remark}
  \newtheorem*{rem*}{\protect\remarkname}
  \theoremstyle{remark}
  \newtheorem*{rems*}{\protect\remarksname}
  \theoremstyle{definition}
  \newtheorem{example}[equation]{\protect\examplename}
  \theoremstyle{definition}
  \newtheorem{algo}[equation]{\protect\algorithmname}
  \theoremstyle{definition}
  \newtheorem{notn}[equation]{\protect\notationname}
  \theoremstyle{definition}
  \newtheorem{setup}[equation]{\protect\setupname}

  \providecommand{\remarksname}{Remarks}
  \providecommand{\algorithmname}{Algorithm}
\providecommand{\notationname}{Notation}
\providecommand{\setupname}{Setup}

\makeatother

\providecommand{\corollaryname}{Corollary}
\providecommand{\definitionname}{Definition}
\providecommand{\examplename}{Example}
\providecommand{\lemmaname}{Lemma}
\providecommand{\propositionname}{Proposition}
\providecommand{\remarkname}{Remark}
\providecommand{\theoremname}{Theorem}

\begin{document}
\title{The Eigenvalues of Hyperoctahedral Descent Operators and Applications to Card-Shuffling}
\author{C.Y. Amy Pang\thanks{amypang@hkbu.edu.hk}}

\maketitle
\maketitle
\global\long\def\tautilde{\tilde{\tau}}%
\global\long\def\T{\mathbf{T}}%
\global\long\def\f{\mathbf{f}}%
\global\long\def\ftilde{\tilde{\f}}%
\global\long\def\g{\mathbf{g}}%
\global\long\def\hatk{\check{K}}%
\global\long\def\sn{\mathfrak{S}_{n}}%
\global\long\def\sk{\mathfrak{S}_{k}}%
\global\long\def\oriftauplus{\orif\tau_{a}^{+}}%
\global\long\def\oriftauminus{\orif\tau_{a}^{-}}%
\global\long\def\oriftautildeplus{\orif\tautilde_{a}^{+}}%
\global\long\def\oriftautildeminus{\orif\tautilde_{a}^{-}}%

\global\long\def\bark{\bar{k}}%
\global\long\def\barp{\bar{p}}%
\global\long\def\bari{\bar{i}}%
\global\long\def\barj{\bar{j}}%
\global\long\def\barlambda{\bar{\lambda}}%
\global\long\def\leftu{\overleftarrow{u}}%
\global\long\def\rightu{\overrightarrow{u}}%

\global\long\def\calsh{\mathcal{S}}%
\global\long\def\calh{\mathcal{H}}%
\global\long\def\calhn{\mathcal{H}_{n}}%
\global\long\def\calhdual{\mathcal{H}^{*}}%
\global\long\def\calhndual{\mathcal{H}_{n}^{*}}%
\global\long\def\calb{\mathcal{B}}%
\global\long\def\calbn{\mathcal{B}_{n}}%
\global\long\def\calbdual{\mathcal{B}^{*}}%
\global\long\def\calbndual{\mathcal{B}_{n}^{*}}%
\global\long\def\calp{\mathcal{P}}%
\global\long\def\barcalp{\bar{\calp}}%
\global\long\def\cala{\mathcal{A}}%

\begin{abstract}
We extend an algebra of Mantaci and Reutenauer, acting on the free
associative algebra, to a vector space of operators acting on all
graded connected Hopf algebras. These operators are convolution products
of certain involutions, which we view as hyperoctahedral variants
of Patras's descent operators. We obtain the eigenvalues and multiplicities
of all our new operators, as well as a basis of eigenvectors for a
subclass akin to Adams operations. We outline how to apply this eigendata
to study Markov chains, and examine in detail the case of card-shuffles
with flips or rotations.
\end{abstract}

\section{Introduction\label{sec:Introduction}}

\textcolor{green}{The use of graded Hopf algebras to study combinatorics
is increasingly common \cite{jonirota,hivertcspolynomialrealisation,qsymisterminal,hopfmonoidgenpermutahedra}.}
These are graded vector spaces $\calh$ with bases indexed by combinatorial
objects, such as trees \cite{lodayroncotrees,cambrianhopfalg}, graphs
\cite{schmitt} or permutations \cite{fqsym}, admitting a product
$m:\calh\otimes\calh\rightarrow\calh$ and coproduct $\Delta:\calh\to\calh\otimes\calh$
that encode respectively how these objects combine and break apart.
The composition $m\circ\Delta$, or more generally the composition
of iterated product and coproduct maps $m^{[a]}\circ\Delta^{[a]}:\calh\rightarrow\calh$,
are the \emph{Adams operations}, studied in \cite[Sec. 4.5]{lodayhomology}
for their connections to Hochschild homology. Their eigenvalues and
multiplicities are obtained in \cite{diagonalisingusinggrh}, and
\textcolor{green}{applied to derive some} combinatorial identities.
The paper \cite{hopfpowerchains} gives a basis of eigenvectors for
$m^{[a]}\circ\Delta^{[a]}$ on free-commutative or cocommutative algebras,
and interprets the matrix of these Adams operations as the transition
probabilities of a Markov chain on a basis of $\calh$. When $\calh$
is the shuffle algebra, this probabilistic interpretation recovers
the famous Gilbert-Shannon-Reeds riffle-shuffle of a deck of cards
\cite{riffleshuffle}: cut the deck into $a$ piles according to the
multinomial distribution, then interleave the piles together so cards
from the same pile stay in the same relative order. The eigenvectors
then lead to bounds for certain probabilities under repeated shuffles.

We may doctor the Adams operations by refining the coproduct by degree:
$m^{[a]}\circ\Delta_{(d_{1},\dots,d_{a})}$ are the \emph{descent
operators} of Patras \cite{descentoperators}, closely related to
Solomon's descent algebra of the symmetric group \cite{solomondescentalg}.
The paper \cite{descentoperatorchains} finds the eigenvalues and
eigenvectors of their linear combinations, which correspond to more
general card shuffles where the deck is cut into $a$ piles according
to other distributions, related to the Tsetlin library models of dynamic
storage allocation \cite{tsetlinlibrary}. Although much is already
known about such shuffles, the Hopf-algebraic viewpoint extends easily
to handle analogous ``shuffling'' Markov chains on trees, graphs
and permutations, simply by changing the Hopf algebra $\calh$ in
the universal eigendata formulas.

The present paper presents analogous results on the new class of \emph{hyperoctahedral
descent operators}, motivated by numerous card-shuffling models where,
after cutting the deck, some piles are flipped over before combining
\cite{bergeronorthogidem,typebshufflecasino,prydephatarfodevalues,boxmovesbellnumbers};
see Corollary \ref{cor:shuffle3step} below for the precise definition.
\textcolor{green}{Our hyperoctahedral descent operators have the form
}$m\circ(f_{1}\otimes\dots\otimes f_{a})\circ\Delta_{(d_{1},\dots,d_{a})}$,
where each $f_{i}$ is the identity or an involution $\tau$ (satisfying
certain conditions). These generalise the Mantaci-Reutenauer algebra
of endomorphisms on the free associative algebra \cite{mantacireutenauer}.
Our main results are 
\begin{itemize}
\item Proposition \ref{prop:compositionlaw}, that the composition of hyperoctahedral
descent operators on a commutative or cocommutative Hopf algebra follows
the multiplication in the Mantaci-Reutenauer algebra;
\item Theorem \ref{thm:evalues}, a uniform expression for the eigenvalues
and multiplicities for any linear combination of such maps;
\item Theorems \ref{thm:flipriffle-evector}, \ref{thm:rotriffleodd-evector}
and \ref{thm:rotriffleeven-evector}, a complete basis for the non-zero
eigenspaces of $m^{[a]}\circ(\id\otimes\tau\otimes\id\otimes\tau\otimes\dots)\circ\Delta^{[a]}$
on cocommutative Hopf algebras. This is one analogue of an Adams operation,
corresponding to the hyperoctahedral riffle-shuffles of \cite{riffleshuffle,bergeronorthogidem}
where every other pile is flipped. 
\end{itemize}
We note that our conditions on $\tau$ allow rotating cards by 180
degrees in place of flipping; the eigenvalues of some such shuffles
were identified in the unpublished manuscript \cite{diaconissteinberg}
using coloured versions of left regular bands. Thus this paper joins
a long list of works calculating eigendata for shuffling models \cite{prydephatarfodmultiplicities,hyperplanewalk,wilsonshuffleevector,randomtorandom,nadiaeigen}.
As in \cite{descentoperatorchains}, our eigenvector formulas apply
to all graded connected Hopf algebras, and hence can analyse similar
``flipped shuffles'' of trees or permutations. We note that Markov
chains on signed permutations appear in the study of genome rearrangement
\cite{biosignedpermutation1,biosignedpermutation2}, to compare the
chromosomes of two species. Further, since the theory of the Mantaci-Reutenauer
algebra extends easily to cyclic operators in place of involutions
\cite[Sec. 6]{mantacireutenauer}, the results here should be generalisable
to the much greater class of ``cyclic descent operators'' without
much complications, diagonalising the coloured shuffles of \cite{colouredeuleriandescentalg}. 

The paper is structured as follows: Section \ref{sec:hyperoctahedral-descentop}
defines the hyperoctahedral descent operators and derives their composition
law and eigenvalues. Section \ref{sec:oriffle} specialises to $m^{[a]}\circ(\iota\otimes\tau\otimes\iota\otimes\tau\otimes\dots)\circ\Delta^{[a]}$,
finding their eigenvectors. This ends the purely algebraic part. Section
\ref{sec:chain} explains minor modifications to \cite{descentoperatorchains},
to associate a Markov chain to each hyperoctahedral descent operator
acting on a combinatorial Hopf algebra. Section \ref{sec:Card-Shuffling}
applies the eigenvector formulas to give the probabilities of certain
statistics under hyperoctahedral riffle-shuffles, with flips or rotations.

Following the related Coxeter nomenclature, we refer to our new hyperoctahedral
theory as ``type B'', and the original descent operators as ``type
A''.

Acknowledgements: We thank Ronan Conlon and Eric Marberg for inspiring
this project, and Marcelo Aguiar, Persi Diaconis, Rafael Gonz\'alez
D'Le\'on and Nadia Lafreni\`{e}re for many helpful comments. SAGE
computer software \cite{sage} was very useful, especially the combinatorial
Hopf algebras coded by Aaron Lauve and Franco Saliola.

The author is supported by the grant RGC-ECS 22300017.

\section{Hyperoctahedral Descent Operators\label{sec:hyperoctahedral-descentop}}

This work concerns \emph{graded connected Hopf algebras} - that is,
a graded vector space $\calh=\bigoplus_{n=0}^{\infty}\calh_{n}$  (over
$\mathbb{R}$) equipped with a linear product map $m:\calh_{i}\otimes\calh_{j}\rightarrow\calh_{i+j}$
and a linear coproduct map $\Delta:\calh_{n}\to\bigoplus_{i=0}^{n}\calh_{i}\otimes\calh_{n-i}$
satisfying certain associativity, coassociativity, and compatibility
axioms \cite{vicreinernotes}. Connectedness means that $\dim\calh_{0}=1$,
so we will identify $\calh_{0}$ with $\mathbb{R}$, writing $\calh_{0}=\sspan\{1\}$.
In many examples where $\calh$ has
a basis indexed by combinatorial objects, $m$ and $\Delta$ can be
interpreted respectively as rules for combining and breaking these
objects, as in our two main examples:
\begin{example}
\label{ex:shufflealg} Fix $N\in\mathbb{N}$, and consider $\mathcal{A:}=\{1,2,\dots,N,\bar{1},\bar{2},\dots,\bar{N}\}$,
where $\bari$ represents a ``negative'' version of $i$. We may
consider the shuffle algebra on this alphabet, whose basis is the
set of words $w_{1}\dots w_{n}$ on $\mathcal{A}$, representing a
deck of cards \textcolor{blue}{with card $w_{1}$ on top, card $w_{2}$
second from the top, and so on, so card $w_{n}$ is at the bottom.}
There are two useful interpretations for $\bari$: either it is
a rotation of card $i$ by $180^{\circ}$, or it is card $i$ flipped
upside down. We may call this the \emph{signed shuffle algebra} to
differentiate it from the (unsigned) shuffle algebra on the alphabet
$\{1,2,\dots,N\}$, which is a quotient under the identification of
$i$ with $\bari$.

\textcolor{blue}{The degree of a word is its number of letters, i.e.
the number of cards in the deck. The product of two words} $u$ and
$v$ \textcolor{blue}{is the sum of all their interleavings} \textcolor{green}{(meaning
all words containing the letters in $u$ and $v$, in the same relative
order as in $u$ and $v$)}, representing the shuffling of two decks.
The coproduct is deconcatenation, or cutting the deck: 
\[
\Delta(w_{1}\dots w_{n})=\sum_{i=0}^{n}w_{1}\dots w_{i}\otimes w_{i+1}\dots w_{n}.
\]
For example:
\begin{align*}
m(1\bar{5}\otimes\bar{3}\otimes\bar{2}) & =1\bar{5}\bar{3}\bar{2}+1\bar{5}\bar{2}\bar{3}+1\bar{2}\bar{5}\bar{3}+\bar{2}1\bar{5}\bar{3}+1\bar{3}\bar{5}\bar{2}+1\bar{3}\bar{2}\bar{5}+1\bar{2}\bar{3}\bar{5}+\bar{2}1\bar{3}\bar{5}\\
 & \phantom{=}+\bar{3}1\bar{5}\bar{2}+\bar{3}1\bar{2}\bar{5}+\bar{3}\bar{2}1\bar{5}+\bar{2}\bar{3}1\bar{5}.
\end{align*}

As Section \ref{subsec:Hyperoctahedral-Riffle-Shuffles} will explain,
the action of hyperoctahedral descent operators on the signed shuffle
algebra induce card-shuffling models.
\end{example}

\begin{example}
\label{ex:freeassalg} On the alphabet $\mathcal{A}$ of the previous
example, we may construct a free associative algebra, which for clarity
we'll term the \emph{signed free associative algebra}. Its basis is
also the set of words on $\mathcal{A}$. The product is concatenation,
representing placing one deck above another, and the coproduct is
``deshuffling'', i.e. 
\begin{align*}
m(u_{1}\dots u_{n}\otimes v_{1}\dots v_{m}) & =u_{1}\dots u_{n}v_{1}\dots v_{m};\\
\Delta(w_{1}\dots w_{n}) & =\sum_{S\subseteq\{1,2,\dots,N\}}\prod_{i\in S}w_{i}\otimes\prod_{i\notin S}w_{i}.
\end{align*}
For example:
\begin{align*}
m(1\bar{5}\otimes\bar{3}\bar{2})=(1\bar{5})(\bar{3}\bar{2}) & =1\bar{5}\bar{3}\bar{2};\\
\Delta(3\bar{1}6) & =\emptyset\otimes3\bar{1}6+3\otimes\bar{1}6+\bar{1}\otimes36+6\otimes3\bar{1}\\
 & \phantom{=}+3\bar{1}\otimes6+36\otimes\bar{1}+\bar{1}6\otimes3+3\bar{1}6\otimes\emptyset.
\end{align*}
As explained five paragraphs below, the signed free associative algebra
is dual to the signed shuffle algebra, which means the eigenvectors
for hyperocatedral descent operators acting on the signed free associative
algebra \textcolor{green}{give probabilistic results about card shuffles
- see Section \ref{subsec:oriffle-descentpeak}.}
\end{example}

Note that the product structures of the above two examples were considered
in \cite{bergeronorthogidem}.

In any Hopf algebra, we abuse notation and write $m$ for the product
of more than two factors, e.g. in the signed shuffle algebra, 
\begin{align*}
m(1\bar{5}\otimes\bar{3}\otimes\bar{2}) & =1\bar{5}\bar{3}\bar{2}+1\bar{5}\bar{2}\bar{3}+1\bar{2}\bar{5}\bar{3}+\bar{2}1\bar{5}\bar{3}+1\bar{3}\bar{5}\bar{2}+1\bar{3}\bar{2}\bar{5}\\
 & \phantom{=}+\bar{3}1\bar{5}\bar{2}+\bar{3}1\bar{2}\bar{5}+\bar{3}\bar{2}1\bar{5}.
\end{align*}
Dually, because of coassociativity, there is a well-defined notion
of splitting a combinatorial object into $l$ parts: $\Delta^{[l]}:\calh\rightarrow\calh^{\otimes l}$
defined inductively by $\Delta^{[2]}=\Delta$, $\Delta^{[l]}=(\Delta^{[l-1]}\otimes\id)\circ\Delta$.
It is useful to constrain the ``sizes of the parts'' by projecting
to graded subspaces after applying $\Delta^{[l]}$. Namely, for a
\emph{weak-composition} $D=(d_{1},\dots,d_{l(D)})$ of $n$ (i.e.
non-negative integers $d_{i}$ summing to $n$), set $\Delta_{D}:\calhn\rightarrow\calh_{d_{1}}\otimes\dots\otimes\calh_{d_{l(D)}}$.
For example, in the signed shuffle algebra, $\Delta_{2,1,1}(1\bar{5}\bar{3}\bar{2})=1\bar{5}\otimes\bar{3}\otimes\bar{2}$
and $\Delta^{[3]}(1\bar{5}\bar{3}\bar{2})=1\bar{5}\otimes\bar{3}\otimes\bar{2}+1\otimes\bar{5}\bar{3}\otimes\bar{2}+1\bar{5}\bar{3}\otimes\emptyset\otimes\bar{2}+\emptyset\otimes\emptyset\otimes1\bar{5}\bar{3}\bar{2}+\dots$
(11 more terms).

A key idea for determining the spectrum and eigenvectors of hyperoctahedral
descent operators is the \emph{primitive} subspace of a Hopf algebra:
$Prim(\calh)=\{x\in\calh|\Delta(x)=1\otimes x+x\otimes1\}$. Because
of the counit axiom, $(Prim\calh)\cap\calhn$ is equivalently characterised
by $\{x\in\calhn|\Delta(x)\subseteq\calh_{0}\otimes\calhn+\calhn\otimes\calh_{0}\}$.
$Prim(\calh)$ is a Lie algebra, meaning that it is closed under the
\emph{Lie bracket}: if $x,y\in Prim(\calh)$, then $[x,y]:=xy-yx\in Prim(\calh)$.

A Hopf algebra $\calh$ is \emph{commutative} if $wz=zw$ for all
$w,z\in\calh$. And $\calh$ is \emph{cocommutative} if, for all $x\in\calh$,
its coproduct $\Delta(x)$ is invariant under the swapping of tensor-factors,
i.e. $\Delta(x)=\sum_{i}w_{i}\otimes z_{i}=\sum_{i}z_{i}\otimes w_{i}$. 

\textcolor{blue}{Given a graded connected Hopf algebra $\calh=\bigoplus_{n\geq0}\calhn$,
the symmetry of the Hopf axioms allows the definition of a Hopf structure
on the (graded) dual vector space $\calhdual:=\oplus_{n\geq0}\calhndual$:
for $f,g\in\calhdual$, set
\[
m(f\otimes g)(x):=(f\otimes g)(\Delta x),\quad\Delta(f)(w\otimes z)=f(wz),
\]
with $x,z,w\in\calh$. (Here, $(f\otimes g)(a\otimes b)=f(a)g(b)$.)
}Note that the dual of a commutative Hopf algebra is cocommutative,
and vice versa.

The signed shuffle algebra is commutative, and is dual to the signed
free associative algebra, which is cocommutative.

\subsection{The Involutions $\tau$ and $\protect\tautilde$\label{subsec:tau}}

\textcolor{green}{Here are various adjectives describing endomorphisms
of a Hopf algebra, of which the last is new and convenient for this
work.}
\begin{defn}
Let $\calh$ be a graded connected Hopf algebra, and $f:\calh\rightarrow\calh$
a function.
\begin{itemize}
\item $f$ is \emph{graded} if $f(\calhn)\subseteq\calhn$;
\item $f$ is an \emph{involution} if $f\circ f=\id$;
\item $f$ is an \emph{algebra morphism} if $f(wz)=f(w)f(z)$ for all $w,z\in\calh$;
\item $f$ is an \emph{algebra antimorphism} if $f(wz)=f(z)f(w)$ for all
$w,z\in\calh$;
\item $f$ is a \emph{coalgebra morphism} if, whenever $\Delta(x)=\sum_{i}w_{i}\otimes z_{i}$,
then $\Delta(f(x))=\sum_{i}f(w_{i})\otimes f(z_{i})$;
\item $f$ is a \emph{coalgebra antimorphism} if, whenever $\Delta(x)=\sum_{i}w_{i}\otimes z_{i}$,
then $\Delta(f(x))=\sum_{i}f(z_{i})\otimes f(w_{i})$;
\item $f$ is a \emph{Hopf morphism} if it is both an algebra morphism and
coalgebra morphism.
\item \emph{$f$ is Hopf ambimorphism} if it is either an algebra morphism
or algebra antimorphism, and also either a coalgebra morphism or coalgebra
antimorphism.
\end{itemize}
\end{defn}

Let $\tau:\calh\rightarrow\calh$ denote an involution that is linear,
graded and a Hopf ambimorphism. Given the definition of ambimorphism,
there are $2\times2=4$ possible combinations for how $\tau$ interacts
with the Hopf structure of $\calh$, and some general results (e.g.
Theorem \ref{thm:evalues} concerning the spectrum of associated descent
operators) apply equally in all four cases. \textcolor{green}{Where
the four possibilities behave differently, we may use $\tau$ and
$\tautilde$ to emphasise the differences, where $\tau$ usually includes
the case of a Hopf morphism and perhaps also other cases.}

Note that, if $\calh$ is commutative, then an algebra morphism is
also an algebra antimorphism; dually, if $\calh$ is cocommutative,
then a coalgebra morphism is also a coalgebra antimorphism.
\begin{example}
Recall the signed shuffle algebra of Example \ref{ex:shufflealg}.
Since this Hopf algebra is commutative, there are two possible types
of Hopf-ambimorphism: we may let $\tau$ be an algebra morphism and
coalgebra morphism, or we may let $\tautilde$ be an algebra morphism
and coalgebra antimorphism. For the card-shuffling applications in
this work, we focus on the involutions from \cite{bergeronorthogidem}:
let $\tau$ model the rotation of a deck by $180^{\circ}$(where $\bari$
denotes a rotated copy of card $i$), and let $\tautilde$ model the
flipping of a deck (where $\bari$ denotes card $i$ facing down).
In terms of words, we define
\begin{align}
\tau(i) & :=\bari;\nonumber \\
\tau(\bari) & :=i;\nonumber \\
\tau(w_{1}\dots w_{n}) & :=\tau(w_{1})\dots\tau(w_{n});\label{eq:tau-shufflealg}
\end{align}
\begin{align}
\tautilde(i) & :=\bari;\nonumber \\
\tautilde(\bari) & :=i;\nonumber \\
\tautilde(w_{1}\dots w_{n}) & :=\tautilde(w_{n})\dots\tautilde(w_{1});\label{eq:tautilde-shufflealg}
\end{align}
and extend linearly. For example, $\tau(3\bar{1}6)=\bar{3}1\bar{6}$
whilst $\tautilde(3\bar{1}6)=\bar{6}1\bar{3}$. The crucial difference
between rotating and flipping a deck is that flipping reverses the
top-to-bottom order of the cards while rotation does not - this is
why rotation is a Hopf morphism but flipping is not.

Equations (\ref{eq:tau-shufflealg}) and (\ref{eq:tautilde-shufflealg})
can also be applied in the signed free associative algebra, then $\tau$
is a Hopf morphism, and $\tautilde$ is an algebra antimorphism and
a coalgebra morphism. (These involutions are actually dual to the
$\tau$ and $\tautilde$ defined by the same equations on the signed
shuffle algebra.)

Further, $\tautilde$ is compatible with the quotient of the signed
free associative algebra to the unsigned free associative algebra;
in other words, we may define 
\begin{equation}
\tautilde(w_{1}\dots w_{n}):=w_{n}\dots w_{1}\label{eq:tautilde-unsignedshufflealg}
\end{equation}
 on the unsigned free associative algebra, which is an algebra antimorphism
and a coalgebra morphism. 
\end{example}

\textcolor{green}{Below is another example that is interesting for
future study.}
\begin{example}
The paper \cite{cambrianhopfalg} defines a Hopf algebra structure
$Camb^{*}$ on Cambrian trees, a type of planar binary tree where
each internal node has either two upward edges and one downward edge,
or two downward edges and one upward edge. Loosely speaking, the product
assembles one tree on top of another (vertical), and the coproduct
divides a tree into left and right (horizontal). Thus horizontal reflection
is an algebra morphism and a coalgebra antimorphism; vertical reflection
is an algebra antimorphism and coalgebra morphism; $180^{\circ}$-rotation
is an algebra antimorphism and a coalgebra antimorphism. Indeed, $180^{\circ}$-rotation
is a composition of the two reflections, and it is easy to check that
composing two commuting involutive Hopf-ambimorphisms always yields
another involutive Hopf-ambimorphism. 
\end{example}

The following notation will be useful:
\begin{defn}
\label{def:invt-ngtg-subspaces} Let $V$ be a vector space and $\tau:V\rightarrow V$
be a linear map. Then the\emph{ $\tau$-invariant part} of $V$, written
$V^{\tau}$, is the eigenspace of eigenvalue 1. The \emph{$\tau$-negating
part} of $V$, written $V^{-\tau}$, is the eigenspace of eigenvalue
-1. 
\end{defn}

Please note the following easy linear algebra lemma:
\begin{lem}
If $\tau$ is an involution on a vector space $V$, then $V=V^{\tau}\oplus V^{-\tau}$.
\end{lem}

\begin{proof}
Clearly $V^{\tau}+V^{-\tau}$ is a direct sum since eigenspaces of
different eigenvalues have trivial intersection. Given any $v\in V$,
we have $v+\tau(v)\in V^{\tau}$ and $v-\tau(v)\in V^{-\tau}$, so
$v=\frac{v+\tau(v)}{2}+\frac{v-\tau(v)}{2}\in V^{\tau}\oplus V^{-\tau}$.
\end{proof}

To derive eigenvalues and eigenvectors, it will be important to consider
the $\tau$-invariant and $\tau$-negating parts of another subspace.
\begin{lem}
If $\calh$ is a graded connected Hopf algebra and $\tau:\calh\rightarrow\calh$
is a coalgebra morphism or coalgebra antimorphism, then $\tau$ preserves
the primitive subspace, i.e. $\tau(Prim(\calh))\subseteq Prim(\calh)$.
\end{lem}

\begin{proof}
If $x\in Prim(\calh)$, then $\Delta(x)=1\otimes x+x\otimes1$. If
$\tau$ is a coalgebra morphism, then $\Delta(\tau(x))=\tau(1)\otimes\tau(x)+\tau(x)\otimes\tau(1)$.
If $\tau$ is a coalgebra antimorphism, then $\Delta(\tau(x))=\tau(x)\otimes\tau(1)+\tau(1)\otimes\tau(x)$.
In both cases, $\tau(x)$ is primitive, by the equivalent characterisation
just before Section \ref{subsec:tau} (applied linearly to the parts
of $x$ in each degree).
\end{proof}
By the same argument, $\tau$ preserves the ``coradical filtration'',
which allows important associated graded operators to be defined in
Lemma \ref{lem:reduce-to-gr} below, for computing the eigenvalues
of hyperoctahedral descent operators.

\subsection{Hyperoctahedral Descent Operators\label{subsec:hdos}}

The elementary hyperoctahedral descent operators are indexed by \emph{signed}
and \emph{tilde-signed weak-compositions}, where some parts are decorated
with $-$ or with $\simeq$ respectively. For example, $(\bar{2},4,0,\bar{2})$
is a signed weak-composition. We do not allow the two types of decorations
to occur together in the same weak-composition. And, unlike in \cite{mantacireutenauer},
we also do not allow parts to have a tilde without a sign, since this
does not \textcolor{green}{correspond natually to an endomorphism
of a Hopf algebra}. These objects are also called bi-coloured compositions
in \cite{boxmovesbellnumbers}.

Given a signed or tilde-signed weak-compositions $D=(d_{1},\dots,d_{l})$,
let $D^{+}$ denote the weak-composition formed from forgetting the
decorations in $D$. We may also use absolute value signs to denote
this, e.g. $D^{+}=(|d_{1}|,\dots,|d_{l}|)$. For example, $(\bar{2},4,0,\bar{2})^{+}=(2,4,0,2)$,
as $|\bar{2}|=2$.
\begin{defn}
\label{def:hdo} Let $\calh$ be a graded Hopf algebra, and $\tau:\calh\rightarrow\calh$
be an involution that is linear, graded and a Hopf morphism. Let $D$
be a signed weak-composition. Then 

\begin{itemize}
\item the \emph{signed refined coproduct} $\Delta_{D(\tau)}$ is the composite:
first apply the refined coproduct $\Delta_{D^{+}}$, then apply $\tau$
to the tensorands corresponding to the decorated parts.
\item the \emph{elementary signed descent operators }are the composites
$m\circ\Delta_{D(\tau)}$.
\item the \emph{signed descent operators} are linear combinations of the
$m\circ\Delta_{D(\tau)}$ over different signed weak-compositions
$D$.
\end{itemize}
All three maps have tilde-signed analogues, by using tilde-signed
weak-compositions, and an involution $\tautilde$ that is a Hopf ambimorphism
but not a Hopf morphism. We refer to both families of operators as
\emph{hyperoctahedral descent operators}. In numerical examples, when
$\tau$ is clear from the context, we often write $m\circ\Delta_{D}$
in place of $m\circ\Delta_{D(\tau)}$.
\end{defn}

For example, $m\circ\Delta_{1,\tilde{\bar{2}}}=m\circ(\id\otimes\tautilde)\circ\Delta_{1,2}$
is an elementary tilde-signed descent operator. In the shuffle algebra,
using $\tautilde$ as defined in (\ref{eq:tautilde-shufflealg}),
\[
m\circ\Delta_{1,\tilde{\bar{2}}}(3\bar{1}6)=m\circ(\id\otimes\tautilde)(3\otimes\bar{1}6)=m(3\otimes\bar{6}1)=3\bar{6}1+\bar{6}31+\bar{6}13.
\]
And $m\circ\Delta_{\bar{2},4,\bar{2}}+m\circ\Delta_{\bar{3},5}=m\circ(\tau\otimes\id\otimes\tau)\circ\Delta_{2,4,2}+m\circ(\tau\otimes\id)\circ\Delta_{3,5}$
is a signed descent operator. 

Note that the dual of $m\circ\Delta_{D(\tau)}$ is $m\circ\Delta_{D(\tau^{*})}$
on the dual Hopf algebra, where $\tau^{*}:\calhdual\rightarrow\calhdual$
denotes the linear-algebraic dual map. This will be important when
considering right eigenfunctions of the associated Markov chains.

If $\calh$ is commutative or cocommutative, then the vector space
of signed descent operators (for a fixed $\tau$) is closed under
composition, and is isomorphic to the Mantaci-Reutenauer algebra.
The vector space of tilde-signed descent operators is also closed
under composition. The case where $\calh$ is the signed free associative
algebra and $\tau,\tautilde$ are as in (\ref{eq:tau-shufflealg})
and (\ref{eq:tautilde-shufflealg}) is part of Theorems 2.2, 3.8 and
Corollary 5.3 of \cite{mantacireutenauer}. Minor modifications suffice
to extend the proof to the general case. Stating the composition formula
requires the following concepts:
\begin{defn}
\label{def:compatible-mtx} Let $D=(d_{1},\dots,d_{l})$ and $D'=(d'_{1},\dots,d'_{l'})$
be two signed weak-compositions or two tilde-signed weak-compositions.
An $l\times l'$ matrix $M=(M_{i,j})$ of integers is called \emph{compatible
with $D$ and $D'$} (the order matters) if the following two conditions
are satisfied:
\begin{enumerate}
\item The row sums of $M^{+}$ are $D^{+}$, and the column sums of $M^{+}$
are $(D')^{+}$, i.e. $\sum_{j=1}^{l'}|M_{i,j}|=|d_{i}|$ for each
$i\in[1,l]$, and $\sum_{i=1}^{l}|M_{i,j}|=|d'_{j}|$ for each $j\in[1,l']$.
\item The sign of $M_{i,j}$ is the product of the signs of $d_{i}$ and
$d'_{j}$.
\end{enumerate}
Let $\Mat(D,D')$ denote the set of matrices compatible with $D$
and $D'$. For $M\in\Mat(D,D')$, let $\wcomp(M)$ denote the signed
weak-composition formed by reading left to right across each row from
the top row to the bottom row 
\[
\wcomp(M)=(M_{1,1},M_{1,2},\dots,M_{1,l'},M_{2,1,},\dots M_{2,l'},\dots,M_{l,1},\dots,M_{l,l'}).
\]
And let $\widetilde{\wcomp_{D}}(M)$ denote the tilde-signed weak-composition
formed by reading each row from the top row to the bottom row, where
the $i$th row is read left to right if $d_{i}$ is positive, and
read right to left if $d_{i}$ is negative (i.e. decorated).

\end{defn}

\begin{example}
Let $D=(\bar{2},\bar{4},1)$ and $D'=(\bar{2},5)$. Then $\Mat(D,D')$
consists of the five matrices below:
\[
\begin{array}{cc}
0 & \bar{2}\\
1 & \bar{3}\\
\bar{1} & 0
\end{array},\begin{array}{cc}
1 & \bar{1}\\
0 & \bar{4}\\
\bar{1} & 0
\end{array},\begin{array}{cc}
1 & \bar{1}\\
1 & \bar{3}\\
\bar{0} & 1
\end{array},\begin{array}{cc}
2 & \bar{0}\\
0 & \bar{4}\\
\bar{0} & 1
\end{array},\begin{array}{cc}
0 & \bar{2}\\
2 & \bar{2}.\\
\bar{0} & 1
\end{array}
\]
If $M$ is the first matrix in this list, then $\wcomp(M)=(0,\bar{2},1,\bar{3},\bar{1},0)$.

Now take $D=(\tilde{\bar{2}},\tilde{\bar{4}},1)$, $D'=(\tilde{\bar{2}},5)$
and 
\[
M=\begin{array}{cc}
0 & \tilde{\bar{2}}\\
1 & \tilde{\bar{3}}\\
\tilde{\bar{1}} & 0
\end{array}\in\Mat(D,D').
\]
Then $\widetilde{\wcomp_{D}}(M)=(\tilde{\bar{2}},0,\tilde{\bar{3}},1,\tilde{\bar{1}},0)$.
\end{example}

\begin{prop}
\label{prop:compositionlaw} Under the setup of Definitions \ref{def:hdo}
and \ref{def:compatible-mtx}:
\begin{enumerate}
\item if $\calh$ is commutative, then
\begin{align*}
\left(m\circ\Delta_{D(\tau)}\right)\circ\left(m\circ\Delta_{D'(\tau)}\right) & =\sum_{M\in\Mat(D',D)}m\circ\Delta_{\wcomp(M)};\\
\left(m\circ\Delta_{D(\tautilde)}\right)\circ\left(m\circ\Delta_{D'(\tautilde)}\right) & =\sum_{M\in\Mat(D',D)}m\circ\Delta_{\widetilde{\wcomp_{D'}}(M)};
\end{align*}
\item if $\calh$ is cocommutative, then
\begin{align*}
\left(m\circ\Delta_{D(\tau)}\right)\circ\left(m\circ\Delta_{D'(\tau)}\right) & =\sum_{M\in\Mat(D,D')}m\circ\Delta_{\wcomp(M)};\\
\left(m\circ\Delta_{D(\tautilde)}\right)\circ\left(m\circ\Delta_{D'(\tautilde)}\right) & =\sum_{M\in\Mat(D,D')}m\circ\Delta_{\widetilde{\wcomp_{D}}(M)}.
\end{align*}
\end{enumerate}
\end{prop}

Composition formulas for more complex hyperoctahedral descent operators
can be obtained by taking linear combinations of the above. For example,
\cite[Th. 4.2]{typebshufflecasino} treats the composition of $\sum_{D}m\circ\Delta_{D(\tautilde)}$
on the unsigned shuffle algebra, where the sum ranges over all tilde-signed
weak-compositions of $a$ parts and with a fixed sequence of $a$
signs. As another example, Proposition \ref{prop:riffle-composition}
will derive the composition formula for hyperoctahedral riffle-shuffles.
\begin{proof}
We follow the argument of \cite{mantacireutenauer}. Let $D=(d_{1},\dots,d_{l})$
and $D'=(d'_{1},\dots,d'_{l'})$ be signed weak-compositions. For
$i\in[1,l]$ and $j\in[1,l']$, define 
\[
f_{i}=\begin{cases}
\id & \text{if }d_{i}\geq0;\\
\tau & \text{if }d_{i}<0;
\end{cases}
\]
\[
g_{j}=\begin{cases}
\id & \text{if }d'_{j}\geq0;\\
\tau & \text{if }d'_{j}<0.
\end{cases}
\]
Then
\begin{align*}
 & \left(m\circ\Delta_{D(\tau)}\right)\circ\left(m\circ\Delta_{D'(\tau)}\right)\\
= & \left(m\circ(f_{1}\otimes\dots\otimes f_{l})\circ\Delta_{D^{+}}\right)\circ\left(m\circ(g_{1}\otimes\dots\otimes g_{l'})\circ\Delta_{D'^{+}}\right)\\
= & \sum m\circ(f_{1}\otimes\dots\otimes f_{l})\circ m\circ\sigma\circ(\Delta_{D_{1}}\otimes\dots\otimes\Delta_{D_{l'}})\circ(g_{1}\otimes\dots\otimes g_{l'})\circ\Delta_{D'^{+}}
\end{align*}
by the compatibility axiom between $m$ and $\Delta$. Here, the sum
is over all $l'$-tuples $(D_{1},\dots,D_{l'})$ of weak-compositions
of $|d'_{1}|,\dots,|d'_{l'}|$ respectively, that satisfy $(D_{1})_{i}+\dots+(D_{l'})_{i}=|d_{i}|$.
And $\sigma:\calh^{\otimes ll'}\rightarrow\calh^{\otimes ll'}$ is
a permutation of the tensorands corresponding to reading a matrix
in columns instead of rows: 
\begin{align*}
\sigma(x_{1}\otimes\dots\otimes x_{ll'}) & =(x_{1}\otimes x_{l+1}\otimes\dots\otimes x_{l(l'-1)+1})\\
 & \qquad\otimes(x_{2}\otimes\dots\otimes x_{l(l'-1)+2})\otimes\dots\\
 & \qquad\otimes(x_{l'}\otimes\dots\otimes x_{ll'}).
\end{align*}
Because $g_{j}$ are coalgebra morphisms, and $f_{i}$ are algebra
morphisms, the above is equal to 
\begin{align}
 & \sum m\circ(f_{1}^{\otimes l'}\otimes\dots\otimes f_{l}^{\otimes l'})\circ\sigma\circ(g_{1}^{\otimes l}\otimes\dots\otimes g_{l'}^{\otimes l})\circ(\Delta_{D_{1}}\otimes\dots\otimes\Delta_{D_{l'}})\circ\Delta_{D'^{+}}\label{eq:compositionlaw1}\\
= & \sum m\circ(f_{1}^{\otimes l'}\otimes\dots\otimes f_{l}^{\otimes l'})\circ\sigma\circ(g_{1}^{\otimes l}\otimes\dots\otimes g_{l'}^{\otimes l})\circ\Delta_{D_{1}\cdot\dots\cdot D_{l'}}\label{eq:compositionlaw2}
\end{align}
where $D_{1}\cdot\dots\cdot D_{l'}$ is a concatenation of weak-compositions.
This last equality follows from coassociativity.

If $\calh$ is commutative, then $m\circ\sigma=m$, so (\ref{eq:compositionlaw2})
is equal to 
\begin{align}
 & \sum m\circ\sigma\circ\left((f_{1}\otimes\dots\otimes f_{l})^{\otimes l'}\right)\circ(g_{1}^{\otimes l}\otimes\dots\otimes g_{l'}^{\otimes l})\circ\Delta_{D_{1}\cdot\dots\cdot D_{l'}}\nonumber \\
= & \sum m\circ\left(f_{1}\circ g_{1}\otimes f_{2}\circ g_{1}\otimes\dots\otimes f_{l}\circ g_{1}\right.\label{eq:compositionlaw3}\\
 & \qquad\quad\otimes f_{1}\circ g_{2}\otimes\dots\otimes f_{l}\circ g_{2}\otimes\dots\nonumber \\
 & \qquad\qquad\left.\otimes f_{1}\circ g_{l'}\otimes\dots\otimes f_{l}\circ g_{l'}\right)\circ\Delta_{D_{1}\cdot\dots\cdot D_{l'}}\nonumber \\
= & \sum_{M\in\Mat(D',D)}m\circ\Delta_{\wcomp(M)}.\nonumber 
\end{align}
If $\calh$ is cocommutative, then 
\[
\sigma\circ\Delta_{D_{1}\cdot\dots\cdot D_{l'}}=\Delta_{((D_{1})_{1},(D_{2})_{1},\dots,(D_{l'})_{1},(D_{1})_{2},\dots,(D_{l'})_{2},\dots,(D_{l})_{1},\dots,(D_{l'})_{l}},
\]
so (\ref{eq:compositionlaw2}) is equal to 

\begin{align*}
 & \sum m\circ(f_{1}^{\otimes l'}\otimes\dots\otimes f_{l}^{\otimes l'})\circ((g_{1}\otimes\dots\otimes g_{l'})^{\otimes l})\circ\sigma\circ\Delta_{D_{1}\cdot\dots\cdot D_{l'}}\\
= & \sum_{M\in\Mat(D,D')}m\circ\left(f_{1}\circ g_{1}\otimes f_{1}\circ g_{2}\otimes\dots\otimes f_{l}\circ g_{l'}\right.\circ\\
 & \qquad\qquad\qquad\otimes f_{2}\circ g_{1}\otimes\dots\otimes f_{2}\circ g_{l'}\otimes\dots\\
 & \qquad\qquad\qquad\left.\otimes f_{l}\circ g_{1}\otimes\dots\otimes f_{l}\circ g_{l'}\right)\circ\Delta_{\wcomp(M)^{+}}\\
= & \sum_{M\in\Mat(D,D')}m\circ\Delta_{\wcomp(M)}.
\end{align*}
Now let $D$, $D'$ be tilde-signed weak-compositions, and let $\tautilde$
be an algebra morphism and a coalgebra antimorphism on a commutative
$\calh$. \textcolor{green}{Performing the analogous calculations,
the only change is that,} in (\ref{eq:compositionlaw1}), each $\Delta_{D_{j}}$
must be calculated using the reverse weak-composition of $D_{j}$
if $d'_{j}$ is tilde-signed. This change persists through to (\ref{eq:compositionlaw3}),
where additionally $f_{j}\circ g_{1}\otimes\dots\otimes f_{j}\circ g_{l'}$
must be replaced by $f_{j}\circ g_{l'}\otimes\dots\otimes f_{j}\circ g_{1}$
whenever $d'_{j}$ is tilde-signed. Thus the result is a sum of $m\circ\Delta_{\widetilde{\wcomp_{D}}(M)}$
terms instead of $m\circ\Delta_{\wcomp(M)}$. The case is similar
when $\tautilde$ is an algebra antimorphism and a coalgebra morphism
on a cocommutative $\calh$.
\end{proof}

\subsection{Eigenvalues\label{subsec:hdo-evalues}}

Just as the eigenvalues of type A descent operators are indexed by
partitions, so the eigenvalues of hyperoctahedral descent operators
are indexed by \emph{double-partitions} of $n$ - that is, pairs of
partitions $\lambda,\barlambda$ with $\sum_{i}\lambda_{i}+\sum_{i}\barlambda_{i}=n$.
\textcolor{green}{(It's more convenient here to view partitions as
multisets of integers, not arranged in any particular order.)} The
values of these eigenvalues themselves are again in terms of set-compositions,
now of signed integers, compatible with $\lambda,\barlambda,D$ as
defined below.
\begin{defn}
Given a set $S$, a \emph{set-composition} of $S$ is a sequence $B_{1,}\dots,B_{l}$
of (possibly empty) disjoint subsets of $S$ with $B_{1}\amalg\dots\amalg B_{l}=S$.
A set-composition is usually written as $B=B_{1}|\dots|B_{l}$, and
the $B_{i}$ are called \emph{blocks}.
\end{defn}

\begin{defn}
Fix a double-partition $\lambda,\barlambda$ of $n$.
\begin{itemize}
\item Given a weak-composition $D$ (not signed) of $n$, a set-composition
$B_{1}\amalg\dots\amalg B_{l(D)}$ of $\{1,2,\dots,l(\lambda)\}\cup\{\bar{1},\bar{2},\dots,\overline{l(\barlambda)}\}$
is called \emph{compatible with $\lambda,\barlambda,D$} if, for
each $i\leq l(D)$, we have $\sum_{j\in B_{i}}\lambda_{j}+\sum_{\bar{j}\in B_{i}}\barlambda_{j}=d_{i}$.
\item Given a signed or tilde-signed weak-composition $D$, let $\beta_{\lambda,\barlambda}^{D}$
denote a signed count of set-compositions compatible with \emph{$\lambda,\barlambda,D^{+}$.}
\textcolor{green}{The sign of each set-composition is the parity of
signed integers within a block corresponding to a decorated part of
D. }
\end{itemize}
\end{defn}

\begin{example}
Let $n=10$ and $D=(\bar{5},4,1)$. The set-compositions compatible
with the double-partition $\lambda=(4,2),\barlambda=(2,1,1)$, $D^{+}$
are $1\bar{2}|\bar{1}2|\bar{3}$ (corresponding to $D^{+}=(\lambda_{1}+\barlambda_{2},\barlambda_{1}+\lambda_{2},\barlambda_{3})$),
$1\bar{3}|\bar{1}2|\bar{2}$, $\bar{1}2\bar{2}|1|\bar{3}$, $\bar{1}2\bar{3}|1|\bar{2}$.
Since $D$ is decorated only in the first part, the sign of a compatible
set-composition is $-1$ if there are an odd number of signed integers
in the first block, as in $1\bar{2}|\bar{1}2|\bar{3}$ and $1\bar{3}|\bar{1}2|\bar{2}$.
For $\bar{1}2\bar{2}|1|\bar{3}$ and $\bar{1}2\bar{3}|1|\bar{2}$,
the sign is $+1$ as their first blocks contain an even number of
signed integers. Thus $\beta_{\lambda,\barlambda}^{D}=-1-1+1+1=0$
in this example. 
\end{example}

\begin{thm}
\label{thm:evalues} Let $\calh=\bigoplus\calhn$ be a graded connected
Hopf algebra over $\mathbb{R}$, $\tau:\calh\rightarrow\calh$ a linear
graded involution that is a Hopf-ambimorphism. Let $D$ be a signed
or tilde-signed weak-composition of $n$. 

\begin{enumerate}
\item The eigenvalues of the associated hyperoctahedral descent operator
$m\circ\Delta_{D(\tau)}:\calhn\rightarrow\calhn$ are $\beta_{\lambda,\barlambda}^{D}$,
as $\lambda,\barlambda$ ranges over all double-partitions of $n$.
\item The multiplicity of the eigenvalue $\beta_{\lambda,\barlambda}^{D}$
is the coefficient of $x_{\lambda,\barlambda}:=x_{\lambda_{1}}\dots x_{\lambda_{l(\lambda)}}\bar{x}_{\barlambda_{1}}\dots\bar{x}_{\barlambda_{l(\barlambda)}}$
in the generating function $\prod_{i}(1-x_{i})^{-b_{i}}\prod_{i}(1-\bar{x}_{i})^{-\bar{b}_{i}}$,
where the numbers $b_{i},\bar{b_{i}}$ satisfy the identities 
\begin{align}
\sum_{n}\dim\calhn x^{n} & =\prod_{i}(1-x^{i})^{-b_{i}-\bar{b_{i}}},\label{eq:multiplicity1}\\
\sum_{n}(\dim\calhn^{\tau}-\dim\calhn^{-\tau})x^{n} & =\prod_{i}(1-x^{i})^{-b_{i}}\prod_{i}(1+x^{i})^{-\bar{b}_{i}}.\label{eq:multiplicity2}
\end{align}
\item Furthermore, if $D_{1},\dots,D_{r}$ are signed weak-compositions
of $n$, or are tilde-signed weak-compositions of $n$, and $a_{1},\dots,a_{r}\in\mathbb{R}$,
then the eigenvalues of $\sum_{i=1}^{r}a_{i}m\circ\Delta_{D_{i}(\tau)}:\calhn\rightarrow\calhn$
are $\sum_{i=1}^{r}a_{i}\beta_{\lambda,\barlambda}^{D_{i}}$, with
multiplicities as above.
\end{enumerate}
\end{thm}

Note that, if $\tau=\id$, then the theorem recovers the type A case
\cite[Th. 3.5]{descentoperatorchains}, as all relevant set-compositions
have positive sign.\textcolor{green}{{} For other $\tau$, as long as
its action on $\calh$ gives the generic case where $b_{i},\bar{b}_{i}$
are all non-zero, the eigenvalues are the same regardless of $\tau$;
only the multiplicities depend on $\tau$. (If some $b_{i}$ or $\bar{b}_{i}$
is zero, then some of the generic eigenvalues might not be achieved.)}
\begin{example}
We compute the eigenvalues of $m\circ\Delta_{(\bar{1},n-1)}$. (Under
the correspondence of \textcolor{green}{Theorem \ref{thm:chain-3step}},
this operator models the shuffle that removes the top card from a
deck of $n$ cards, rotates it by 180 degrees, then reinserts it in
a uniformly chosen position.) For a double-partition $\lambda,\barlambda$,
the parts of size 1 in $\lambda$ or $\barlambda$ are in bijection
with set-compositions compatible with $\lambda,\barlambda,(1,n-1)$,
by considering the first block. A part of size 1 in $\lambda$ induces
a positive sign, whereas a part from $\barlambda$ induces a negative
sign. Hence, if $1(\lambda)$ denotes the number of parts of size
1 in $\lambda$, we have $\beta_{\lambda,\barlambda}^{(\bar{1},n-1)}=1(\lambda)-1(\barlambda)$.
Since $1(\lambda)$ and $1(\barlambda)$ can take any two values
in $\{0,1,\dots,n\}$ that sum to $\{0,1,\dots,n-2\}\cup\{n\}$, the
generic eigenvalues are $\{-n\}\cup\{-n+2,-n+1,\dots,n-2\}\cup\{n\}$.
(To obtain the eigenvalues of the card shuffle, we should divide by
$n$, i.e. they are $\frac{k}{n}$ for $k\in\{-n\}\cup\{-n+2,-n+1,\dots,n-2\}\cup\{n\}$.)
\end{example}

\begin{example}
\label{ex:evalue-oriftau2minus} We compute the eigenvalues of $m\circ(\tautilde\otimes\id)\circ\Delta=\sum_{i=0}^{n}m\circ\Delta_{(\tilde{\bar{i}},n-i)}$;
this operator is $\orif\tautilde_{2}^{-}$ in the notation of Section
\ref{sec:oriffle}, and represents the shuffle that cuts the deck
binomially, flips the first pile upside-down, then interleave the
piles together. The paper \cite{prydephatarfodevalues} previously
computed these eigenvalues for a related inverse shuffle,\textcolor{green}{{}
that corresponds to the case of} the unsigned free associative algebra,
with $\tautilde$ as in (\ref{eq:tautilde-unsignedshufflealg}).

Fix a double-partition $\lambda,\barlambda$. Every set-composition
of $\{1,2,\dots,l(\lambda)\}\cup\{\bar{1},\bar{2},\dots,\overline{l(\barlambda)}\}$
into two blocks is compatible with $\lambda,\barlambda,(\tilde{\bar{i}},n-i)$
for exactly one $i$, namely $i=\sum_{j\in B_{1}}\lambda_{j}+\sum_{\bar{j}\in B_{1}}\barlambda_{j}$.
The associated sign is the parity of signed integers in the first
block. If $\barlambda\neq\emptyset$, then there is a sign-reversing
involution on these set-compositions: if $\bar{1}\in B_{1}$, move
it to $B_{2}$, and if $\bar{1}\in B_{2}$, move it to $B_{1}$. Hence
the signed count $\sum_{i=0}^{n}\beta_{\lambda,\barlambda}^{(\tilde{\bar{i}},n-i)}$
is 0. If $\bar{\lambda}=\emptyset$, then the associated eigenvalue
$\sum_{i=0}^{n}\beta_{\lambda,\barlambda}^{(\tilde{\bar{i}},n-i)}$
is the number of set-compositions of $\{1,2,\dots,l(\lambda)\}$ into
two blocks, which is $2^{l(\lambda)}$. Since $l(\lambda)$ can take
any value in $\{1,2,\dots,n\}$ when $\bar{\lambda}=\emptyset$, the
generic eigenvalues are $2,4,\dots2^{n}$ and 0. (The eigenvalues
of the card shuffle require dividing by $2^{n}$, i.e. they are $2^{-k}$
for $k\in\{0,1,\dots,n-1\}$, and 0.)
\end{example}

\begin{example}
\label{ex:randomflip-evalue} Consider the following new type of card
shuffle: deconcatenate the deck into two piles according to the binomial
distribution, then flip a coin for each pile to decide whether or
not to flip it before interleaving the piles together. According to
Theorem \ref{thm:chain-3step}, the corresponding hyperoctahedral
descent operator (up to scaling by $2^{n+2}$) is $\sum m\circ\Delta_{D}$,
summing over all tilde-signed weak-compositions $D$ of $n$ with
two parts. So its eigenvalues are $\sum_{i=0}^{n}\beta_{\lambda,\barlambda}^{(i,n-i)}+\beta_{\lambda,\barlambda}^{(\bar{i},n-i)}+\beta_{\lambda,\barlambda}^{(i,\overline{n-i})}+\beta_{\lambda,\barlambda}^{(\bar{i},\overline{n-i})}$.
Each of the $2^{l(\lambda)+l(\barlambda)}$ set-compositions of $\{1,2,\dots,l(\lambda)\}\cup\{\bar{1},\bar{2},\dots,\overline{l(\barlambda)}\}$
into two blocks contributes $+1$ to $\beta_{\lambda,\barlambda}^{(i,n-i)}$
and $(-1)^{l(\barlambda)}$ to $\beta_{\lambda,\barlambda}^{(\bar{i},\overline{n-i})}$,
where $i=\sum_{j\in B_{1}}\lambda_{j}+\sum_{\bar{j}\in B_{1}}\barlambda_{j}$.
And, by a sign-reversing involution as in Example \ref{ex:evalue-oriftau2minus},
$\sum_{i=0}^{n}\beta_{\lambda,\barlambda}^{(\bar{i},n-i)}=\sum_{i=0}^{n}\beta_{\lambda,\barlambda}^{(i,\overline{n-i})}=0$
if $\barlambda\neq\emptyset$, and else is $2^{l(\lambda)}$. Hence
the eigenvalues are 
\[
\sum\beta_{\lambda,\barlambda}^{D}=\begin{cases}
2^{l(\lambda)+2} & \text{if }l(\barlambda)=0;\\
2^{l(\lambda)+l(\barlambda)+1} & \text{if }l(\barlambda)>0\text{ is even};\\
0 & \text{if }l(\barlambda)\text{ is odd}.
\end{cases}
\]
So the generic eigenvalues are $8,16,\dots,2^{n+2}$ and 0, and for
the card shuffle they are $2^{-k}$ for $k\in\{0,1,\dots,n-1\}$,
and 0.
\end{example}

\begin{rems*}
$ $
\begin{enumerate}[label=\arabic*.]
\item Note that, as for the type A case, the eigenvalues $\beta_{\lambda,\barlambda}^{D}$
may coincide for different choices of $\lambda$ and $\barlambda$.
For example, if $D=(\bar{1},n-1)$, then $\beta_{\lambda,\barlambda}^{D}=0$
whenever neither $\lambda$ nor $\barlambda$ has a part of size 1. 
\item Unlike the type A case, the present theorem does not claim that $m\circ\Delta_{D(\tau)}$
is diagonalisable on commutative or cocommutative Hopf algebras. Sage
computations show that this is in fact false.\textcolor{green}{{} }The
diagonalisability proof in \cite[Sec. 3.2]{descentoperatorchains}
for type A descent operators does not extend to hyperoctahedral descent
operators because the matrices involved there will gain negative entries
when extended to type B, so Perron-Frobenius will not apply.
\item The proof below can generalise to give the spectrum of convolution
products $\T_{1}*\dots*\T_{l}:=m\circ(\T_{1}\otimes\dots\otimes\T_{l})\circ\Delta^{[l]}$,
where the $\T_{i}$ are simultaneously diagonalisable graded Hopf
ambimorphisms. \textcolor{green}{Then the eigenvalues $\beta_{\lambda}$
would be a weighted count of set-compositions, where the weights are
the appropriate products of eigenvalues of the $\T_{i}$.}
\end{enumerate}
\end{rems*}

The two main ideas of the proof, as with \cites[Th. 3]{diagonalisingusinggrh}[Th. 3.5]{descentoperatorchains},
are:
\begin{itemize}
\item reduce to the cocommutative case by working in $(\gr\calh)^{*}$,
the dual of the associated graded Hopf algebra with respect to the
coradical filtration (see \cite[Sec. 1.3]{diagonalisingusinggrh}
for the definitions);
\item examine the action of $\gr(m \circ \Delta_{D(\tau)})^{*}$ on a Poincare-Birkhoff-Witt
basis, i.e. on products of primitive elements.
\end{itemize}
\textcolor{green}{The first part requires the following lemma:}
\begin{lem}
\label{lem:reduce-to-gr} Under the conditions of Theorem \ref{thm:evalues},

\begin{enumerate}
\item The dual associated graded map is 
\[
\left(\gr(m\circ\Delta_{D(\tau)}:\calhn\rightarrow\calhn)\right)^{*}=m\circ\Delta_{D((\gr\tau)^{*})}:(\gr\calh)_{n}^{*}\rightarrow(\gr\calh)_{n}^{*}.
\]
\item $(\gr\tau)^{*}$ is an involution.
\item The dimensions of the fixed subspaces are related by $\dim\calh_{n}^{\tau}=\dim(\gr\calh^{*})_{n}^{(\gr\tau)^{*}}$
and $\dim\calh_{n}^{-\tau}=\dim(\gr\calh^{*})_{n}^{-(\gr\tau)^{*}}$.
\item $(\gr\tau)^{*}$ is a Hopf ambimorphism.
\end{enumerate}
\end{lem}

\begin{proof}
$ $

\begin{enumerate}
\item Since $\tau$ is a coalgebra morphism or antimorphism, $\tau$ preserves
the coradical filtration. Hence $\gr\tau:\gr\calh\rightarrow\gr\calh$
is well-defined. Then, taking the associated graded map is functorial
and so preserves convolution products. The same is true for dualising.
\item Taking the associated graded map is functorial and so preserves compositions.
Hence $\tau\circ\tau=\id$ means $(\gr\tau)\circ(\gr\tau)=\id$, and
dualising this shows that $(\gr\tau)^{*}$ is an involution.
\item \textcolor{green}{As noted in \cite{diagonalisingusinggrh},} $\tau$
and $\gr(\tau)$ have the same spectrum, and so does $(\gr\tau)^{*}$.
Since $\tau$ and $(\gr\tau)^{*}$ are involutions, they are diagonalisable,
and so their spectrum determines their eigenspace dimensions.
\item It follows from the definition of the Hopf structure on $\gr(\calh)$
that, if $\tau$ is an algebra morphism (resp. antimorphism), then
so is $\gr(\tau)$, and then $(\gr\tau)^{*}$ is a coalgebra morphism
(resp. antimorphism). Similarly, if $\tau$ is a coalgebra morphism
(resp. antimorphism), then $(\gr\tau)^{*}$ is an algebra morphism
(resp. antimorphism).
\end{enumerate}
\end{proof}
Any coradical-filtration preserving map and its dual associated graded
map have the same spectrum, and $(\gr\calh)^{*}$ is cocommutative
\cites[Th. 11.2.5.a]{sweedler}[Prop. 1.6]{grhiscommutative}. So the
Lemma above reduces the proof to the case when $\calh$ is cocommutative,
where its structure is well-understood. Indeed, \textcolor{blue}{by
the Cartier-Milnor-Moore theorem \cite[Th. 3.8.1]{cmm}, a graded
connected cocommutative Hopf algebra $\calh$ is the universal enveloping
algebra of its subspace of primitives. Consequently, $\calh$ has
a Poincare-Birkhoff-Witt (PBW) basis: if $(\calp,\preceq)$ is an
ordered basis of the primitive subspace of $\calh$, then $\{p_{1}\dots p_{k}|k\in\mathbb{N},p_{1}\preceq\dots\preceq p_{k}\in\calp\}$
is a basis of $\calh$. The basis element $p_{1}\dots p_{k}$ has
}\textcolor{blue}{\emph{length}}\textcolor{blue}{{} $k$. }The key to
the proof is the action of hyperoctahedral descent operators on this
PBW basis, as stated in the following Triangularity Lemma (a substitute
for the Symmetrisation Lemma of \cite[Lem. 3.8, 3.9]{descentoperatorchains}). 

\begin{notn}\label{notn:pbwbasis} Let $\calh$ be a graded connected
Hopf algebra that is cocommutative, and let $\tau:\calh\rightarrow\calh$
be a linear graded involution that is a Hopf-ambimorphism. Let $(\calp,\preceq)$
and $(\barcalp,\preceq)$ be ordered bases of $(Prim\calh)^{\tau}$
and $(Prim\calh)^{-\tau}$ respectively. Then $\calp\cup\barcalp$,
with the ``concatenation'' order of $p\preceq\barp$ for all $p\in\calp$,
$\barp\in\barcalp$, is an ordered basis of $Prim\calh$, and thus
can be used to construct a PBW basis of $\calh$. Let $p_{1},\dots,p_{k}\in\calp$,
$\barp_{1},\dots,\barp_{\bark}\in\barcalp$ with $p_{1}\preceq\dots\preceq p_{k}$,
$\barp_{1}\preceq\dots\preceq\barp_{\bark}$, and let $\lambda=(\deg p_{1},\dots,\deg p_{k})$, $\barlambda=(\deg\barp_{1},\dots,\deg\barp_{\bark})$.\end{notn}
\begin{lem}[Triangularity Lemma]
\label{lem:trilemma}  Under the setup of Notation \ref{notn:pbwbasis}:
\[
m\circ\Delta_{D(\tau)}(p_{1}\dots p_{k}\barp_{1}\dots\barp_{\bark})=\beta_{\lambda,\barlambda}^{D}p_{1}\dots p_{k}\barp_{1}\dots\barp_{\bark}+\mbox{PBW-basis elements of length less than }k+\bark.
\]
Consequently, relative to this PBW basis, the matrix for $m\circ\Delta_{D(\tau)}$
is triangular with $\beta_{\lambda,\barlambda}^{D}$ as its diagonal
entries, and hence $\beta_{\lambda,\barlambda}^{D}$ are the eigenvalues.
\end{lem}

\begin{proof}
By coassociativity,
\[
\Delta_{D^{+}}(p_{1}\dots p_{k}\barp_{1}\dots\barp_{\bark})=\sum_{B}\left(\prod_{i\in B_{1}}p_{i}\prod_{\bari\in B_{1}}\barp_{\bari}\right)\otimes\dots\otimes\left(\prod_{i\in B_{l(D)}}p_{i}\prod_{\bari\in B_{l(D)}}\barp_{\bari}\right),
\]
 where the sum is over all set-compositions $B$ compatible with $\lambda,\barlambda,D$.
Hence
\begin{equation}
m\circ\Delta_{D(\tau)}(p_{1}\dots p_{k}\barp_{1}\dots\barp_{\bark})=\sum_{B_{1},\dots,B_{l(D)}}\pm\left(\prod_{i\in B_{1}}p_{i}\prod_{\bari\in B_{1}}\barp_{\bari}\right)\dots\left(\prod_{i\in B_{l(D)}}p_{i}\prod_{\bari\in B_{l(D)}}\barp_{\bari}\right),\label{eq:descent-operator-on-product-of-primitives}
\end{equation}
where the sign for each summand is the parity of elements of $\barcalp$
in decorated parts. (If $\tau$ is an algebra antimorphism, then for
each decorated part $d_{i}$, the products in the $i$th bracket above
are taken in the reversed order.) Each summand is a product of $p_{1},\dots,p_{k},\barp_{1}\dots,\barp_{\bark}$
in some order, so the PBW straightening algorithm \cite[Lem. III.3.9]{pbwref}
rewrites each summand (excluding the sign) as $p_{1}\dots p_{k}\barp_{1}\dots\barp_{\bark}+$
terms of length less than $k+\bark$. Thus the coefficient of the
highest length term $p_{1}\dots p_{k}\barp_{1}\dots\barp_{\bark}$
in $m\circ\Delta_{D(\tau)}(p_{1}\dots p_{k}\barp_{1}\dots\barp_{\bark})$
is the signed number of summands, i.e. the signed number of set-compositions
compatible with $\lambda,\barlambda,D$.
\end{proof}
\textcolor{green}{To continue the proof of Theorem \ref{thm:evalues}:}
according to Lemma \ref{lem:trilemma}, the matrices of $m\circ\Delta_{D(\tau)}$
are simultaneously triangularisable with $\beta_{\lambda,\barlambda}^{D}$
on the diagonal, hence the eigenvalues of sums of $m\circ\Delta_{D(\tau)}$
are sums of these diagonal entries. The multiplicity of $\beta_{\lambda,\barlambda}^{D}$
is the number of multiset pairs $\{p_{1},\dots,p_{k}\}\subseteq\calp$,
$\{\barp_{1}\dots,\barp_{\bark}\}\subseteq\barcalp$ whose degrees
are given by $\lambda,\barlambda$. If $b_{i}=\dim(Prim\calh)_{i}^{\tau}=|\calp\cap\calh_{i}|$
and $\bar{b}_{i}=\dim(Prim\calh)_{i}^{-\tau}=|\barcalp\cap\calh_{i}|$,
then the multiplicities of $\beta_{\lambda,\barlambda}^{D}$ are
given by the generating function in Theorem \ref{thm:evalues}.ii.
So it remains to show that the sequences $b_{i}$ and $\bar{b}_{i}$
are determined by the identities (\ref{eq:multiplicity1}) and (\ref{eq:multiplicity2}).

To see (\ref{eq:multiplicity1}) (same argument as in the type A case):
the PBW basis elements of degree $n$ are precisely the products over
a multiset in $\calp$ and a multiset in $\barcalp$, whose degrees
total $n$.

To see (\ref{eq:multiplicity2}): note that the PBW basis element
$p_{1}\dots p_{k}\barp_{1}\dots\barp_{\bark}$ is $\tau$-invariant
if $\bark$ is even, and $\tau$-negating if $\bark$ is odd. \textcolor{green}{Make
a signed enumeration of such products, where each $\barp\in\barcalp$
is signed, and all $p\in\calp$ are unsigned, so that $p_{1}\dots p_{k}\barp_{1}\dots\barp_{\bark}$
is $\tau$-negating if it is signed, and $\tau$-invariant if it is
unsigned. Hence the coefficients are $\dim\calhn^{\tau}-\dim\calhn^{-\tau}$.}

\section{The Hyperoctahedral-Riffle-Shuffle Operators\label{sec:oriffle}}

This section focuses on the following four families of operators,
whose rescaling (division by $a^{n}$) corresponds to $a$-handed
hyperocatahedral riffle-shuffles of decks of $n$ cards, as studied
in \cite{bergeronorthogidem,fulmanbrauercomplex,fulmancyclestructure}.
For a linear graded involution $\tau$ that is a Hopf ambimorphism,
define 
\begin{align}
\oriftauplus & :=\begin{cases}
m\circ(\id\otimes\tau\otimes\id\otimes\tau\otimes\dots\otimes\id)\circ\Delta^{[a]} & \mbox{if }a\mbox{ is odd};\\
m\circ(\id\otimes\tau\otimes\id\otimes\tau\otimes\dots\otimes\tau)\circ\Delta^{[a]} & \mbox{if }a\mbox{ is even};
\end{cases}\nonumber \\
\oriftauminus & :=\begin{cases}
m\circ(\tau\otimes\id\otimes\tau\otimes\id\otimes\dots\otimes\tau)\circ\Delta^{[a]} & \mbox{if }a\mbox{ is odd};\\
m\circ(\tau\otimes\id\otimes\tau\otimes\id\otimes\dots\otimes\id)\circ\Delta^{[a]} & \mbox{if }a\mbox{ is even}.
\end{cases}\label{eq:oriffle-def}
\end{align}
In Section \ref{subsec:oriffle-evector} we will write $\tautilde$
in place of $\tau$ for an algebra antimorphism, \textcolor{green}{to
more clearly distinguish it from the case of an algebra morphism.}

\subsection{Compositions and Properties}

We apply Proposition \ref{prop:compositionlaw} to give a composition
rule for the above operators; a special case appeared in \cite{typebshufflecasino}.
\begin{prop}
\label{prop:riffle-composition} Let $\calh$ be a graded connected
Hopf algebra, and $\tau:\calh\rightarrow\calh$ be a linear graded
involution that is a Hopf ambimorphism. Under any one of these three
conditions:
\begin{enumerate}
\item $\calh$ is commutative and $a$ is odd,
\item $\calh$ is cocommutative and $b$ is odd, 
\item $\calh$ is commutative or cocommutative, and $\tau$ is not a Hopf
morphism,
\end{enumerate}
the operators in (\ref{eq:oriffle-def}) compose as follows: 
\begin{align*}
\oriftauplus\circ\orif\tau_{b}^{+} & =\oriftauminus\circ\orif\tau_{b}^{-}=\orif\tau_{ab}^{+};\\
\oriftauplus\circ\orif\tau_{b}^{-} & =\oriftauminus\circ\orif\tau_{b}^{+}=\orif\tau_{ab}^{-}.
\end{align*}

\end{prop}

\begin{proof}
View $\oriftauplus$ (resp. $\oriftauminus$) as $\sum m\circ\Delta_{D(\tau)}$
over all $D$ with $a$ parts, where \textcolor{green}{even parts}
(resp. \textcolor{green}{odd parts}) are negative. Then, on a commutative
algebra $\calh$, where $\tau$ is a Hopf morphism, Proposition \ref{prop:compositionlaw}.i
gives $\oriftauplus\circ\orif\tau_{b}^{+}=\oriftauminus\circ\orif\tau_{b}^{-}=\sum_{M}m\circ\Delta_{\wcomp(M)}$,
over all $b\times a$ matrices $M$ whose entries have the alternating
sign pattern 
\[
\begin{array}{cccc}
+ & - & + & \dots\\
- & + & -\\
+ & - & \ddots\\
\vdots
\end{array}.
\]
If $a$ is odd, then the first
row of this matrix ends with a positive entry, which is followed by
a negative entry at the start of the second row when computing $\wcomp(M)$.
Similar considerations for other rows shows that $\wcomp(M)$ exactly
runs through all signed weak-compositions with $ba$ parts whose \textcolor{green}{even
parts} are negative. If instead $\tau$ is a coalgebra antimorphism,
then the \textcolor{green}{even rows} must be read right to left when
computing $\widetilde{\wcomp_{D}}(M)$. Thus, if $a$ is odd (resp.
even), then the positive (resp. negative) entry at the end of the
first row is followed by a negative (resp. positive) entry at the
end of the second, so $\widetilde{\wcomp_{D}}(M)$ also runs through
all signed weak-compositions with $ba$ parts whose \textcolor{green}{even
parts} are negative. 

The other cases are similar.
\end{proof}
Next, we specialise Theorem \ref{thm:evalues} to obtain the spectrum
of the hyperoctahedral-riffle-shuffle operators:
\begin{prop}
\label{prop:oriffle-evalues} Let $\calh$ be a graded connected Hopf
algebra, and $\tau:\calh\rightarrow\calh$ be a linear graded involution
that is a Hopf ambimorphism. Consider the operators in  (\ref{eq:oriffle-def})
acting on $\calh_{n}$. Write $[g]f$ to mean the cofficient of the
monomial $g$ in the power series $f$. Then: 
\begin{itemize}
\item for even $a$, the eigenvalues of $\oriftauplus$ and $\oriftauminus$
are $a^{l}$, with multiplicity $[x^{l}y^{n}]\prod_{i}(1-xy^{i})^{-b_{i}}$;
and 0;
\item for odd $a$, the eigenvalues of $\oriftauplus$ are $a^{l}$, with
multiplicity $[x^{l}y^{n}]\prod_{i}(1-xy^{i})^{-b_{i}}\prod_{i}(1-y^{i})^{-\bar{b}_{i}}$;
\item for odd $a$, the eigenvalues of $\oriftauminus$ are:
\begin{itemize}
\item $a^{l}$, with multiplicity $[x^{l}y^{n}]\prod_{i}(1-xy^{i})^{-b_{i}}\frac{1}{2}\left(\prod_{i}(1+y^{i})^{-\bar{b}_{i}}+\prod_{i}(1-y^{i})^{-\bar{b}_{i}}\right)$;
\item $-a^{l}$, with multiplicity $[x^{l}y^{n}]\prod_{i}(1-xy^{i})^{-b_{i}}\frac{1}{2}\left(\prod_{i}(1+y^{i})^{-\bar{b}_{i}}-\prod_{i}(1-y^{i})^{-\bar{b}_{i}}\right)$.
\end{itemize}
\end{itemize}
\end{prop}

\begin{proof}
We generalise the ideas in Example \ref{ex:evalue-oriftau2minus}.
Let $\beta_{\lambda,\barlambda}^{a,+}$ (resp. $\beta_{\lambda,\barlambda}^{a,-}$)
denote the eigenvalue of $\oriftauplus$ (resp. $\oriftauminus$)
corresponding to the double-partition $\lambda,\barlambda$; this
is the signed count of set-compositions $B$ of $\{1,2,\dots,l(\lambda)\}\cup\{\bar{1},\bar{2},\dots,l(\barlambda)\}$
into $a$ blocks, where the sign is the parity of signed integers
in the \textcolor{green}{even (resp. odd)} blocks. First suppose $a$
is even. If $\barlambda\neq\emptyset$, then there is a sign-reversing
involution on these set-compositions: if $\bar{1}\in B_{2i-1}$ for
some $i$, move it to $B_{2i}$, and if $\bar{1}\in B_{2i}$, move
it to $B_{2i-1}$. Hence $\beta_{\lambda,\barlambda}^{a,+}=\beta_{\lambda,\barlambda}^{a,-}=0$
when $\barlambda\neq\emptyset$. \textcolor{green}{If} $\bar{\lambda}=\emptyset$,
then all relevant set-compositions have positive sign, so $\beta_{\lambda,\barlambda}^{a,+}=\beta_{\lambda,\barlambda}^{a,-}=a^{l(\lambda)}$. 

When $a$ is odd, the above signed involution is still defined when
$\barlambda\neq\emptyset$ and $\bar{1}\not\in B_{a}$. So, modify
the signed involution to move the smallest $\bar{j}$ such that $\bar{j}\not\in B_{a}$.
This will be defined when $\barlambda\neq\emptyset$ and $\bar{1},\bar{2},\dots,l(\barlambda)$
are not all in $B_{a}$. Thus the only contributions to $\beta_{\lambda,\barlambda}^{a,+}$
or $\beta_{\lambda,\barlambda}^{a,-}$ come from set-compositions
where $\bar{1},\bar{2},\dots,l(\barlambda)$ are all in $B_{a}$
- these are equivalent to the $a^{l(\lambda)}$ set-compositions of
$\{1,2,\dots,l(\lambda)\}$ into $a$ blocks, and the sign is $(-1)^{l(\barlambda)}$
for $\beta_{\lambda,\barlambda}^{a,-}$ , and $+1$ always for
$\beta_{\lambda,\barlambda}^{a,+}$. 

To see the multiplicities: make the substitution $x_{i}\mapsto xy^{i}$,
$\bar{x}_{i}\mapsto y^{i}$ in Theorem \ref{thm:evalues}.ii, so $x$
tracks the number of positive parts and $y$ tracks the total degree.
Then \textcolor{green}{$[x^{l}y^{n}]f(x,y)=\sum[x_{\lambda,\barlambda}]f(x_{1},x_{2},\dots,\bar{x}_{1},\bar{x}_{2},\dots)$}
summing over all double-partitions $\lambda,\barlambda$ of $n$
with $l$ parts - this handles the case of $\oriftauplus$ for odd
$a$. For $\oriftauminus$, the substitution $x_{i}\mapsto xy^{i}$,
$\bar{x}_{i}\mapsto-y^{i}$ into Theorem \ref{thm:evalues}.ii would
\textcolor{green}{introduce a sign when $\barlambda$ has an odd
number of parts:} \textcolor{green}{$[x^{l}y^{n}]f(x,y)=\sum(-1)^{l(\barlambda)}[x_{\lambda,\barlambda}]f(x_{1},x_{2},\dots,\bar{x}_{1},\bar{x}_{2},\dots)$.}
So, to sum only the coefficients of $x_{\lambda,\barlambda}$ when
$\barlambda$ has an even number of parts, \textcolor{green}{we can
average the signed and unsigned substitutions, and similarly take
their difference to isolate the coefficients }of $x_{\lambda,\barlambda}$
for $\barlambda$ with an odd number of parts.
\end{proof}

\subsection{Eigenvectors\label{subsec:oriffle-evector}}

We give a basis of eigenvectors for the hyperoctahedral-riffle-shuffle
operators. These will aid in computing the expectations of certain
statistics \textcolor{green}{under card-shuffling. }The eigenvector
formulas depend on whether the involution is an algebra morphism (Theorem
\ref{thm:rotriffleodd-evector}, written $\tau$) or antimorphism
(Theorem \ref{thm:flipriffle-evector}, written $\tautilde$). The
antimorphism case is easier, so we begin there. Let $\sk$ denote
the symmetric group on $k$ objects.

\begin{thm}
\label{thm:flipriffle-evector} Let $\calh$ be a graded connected
Hopf algebra, and $\tautilde:\calh\rightarrow\calh$ a linear graded
involution that is an algebra antimorphism and a coalgebra morphism
or antimorphism. 

\begin{enumerate}
\item Take $p_{1},\dots,p_{k}\in(Prim\calh)^{\tautilde}$ and $\barp_{1},\dots,\barp_{\bark}\in(Prim\calh)^{-\tautilde}$(where
$k$ or $\bark$ may be zero).

\begin{enumerate}
\item If $a$ is even, then 
\begin{align*}
\left(\sum_{\sigma\in\sk}p_{\sigma(1)}\dots p_{\sigma(k)}\right)\barp_{1}\dots\barp_{\bark} & \mbox{ is an eigenvector of }\oriftautildeplus\mbox{, with eigenvalue }\begin{cases}
a^{k} & \mbox{if }\bark=0;\\
0 & \mbox{if }\bark>0;
\end{cases}\\
\barp_{1}\dots\barp_{\bark}\left(\sum_{\sigma\in\sk}p_{\sigma(1)}\dots p_{\sigma(k)}\right) & \mbox{is an eigenvector of }\oriftautildeminus\mbox{, with eigenvalue }\begin{cases}
a^{k} & \mbox{if }\bark=0;\\
0 & \mbox{if }\bark>0.
\end{cases}
\end{align*}
\item If $a$ is odd, then 
\[
\barp_{1}\dots\barp_{\bark}\left(\sum_{\sigma\in\sk}p_{\sigma(1)}\dots p_{\sigma(k)}\right)\mbox{ and }\left(\sum_{\sigma\in\sk}p_{\sigma(1)}\dots p_{\sigma(k)}\right)\barp_{1}\dots\barp_{\bark}
\]
are eigenvectors of $\oriftautildeplus$, with eigenvalue $a^{k}$.
\item If $a$ is odd, then
\[
\barp_{1}\dots\barp_{\bark}\left(\sum_{\sigma\in\sk}p_{\sigma(1)}\dots p_{\sigma(k)}\right)+\left(\sum_{\sigma\in\sk}p_{\sigma(1)}\dots p_{\sigma(k)}\right)\barp_{1}\dots\barp_{\bark}
\]
is an eigenvector of $\oriftautildeminus$, for odd $a$, with eigenvalue
$(-1)^{\bark}a^{k}$.
\end{enumerate}
\item If $\calp,\barcalp$ are ordered bases of $(Prim\calh)^{\tautilde},(Prim\calh)^{-\tautilde}$
respectively, then\textcolor{green}{{} the vectors of any fixed format
in i above}, over all choices of $p_{1}\preceq\dots\preceq p_{k}\in\calp$
and $\barp_{1}\preceq\dots\preceq\barp_{\bark}\in\barcalp$ (\textcolor{green}{allowing
$k$ and $\bark$ to vary}), are linearly independent.
\item Furthermore, if $\calh$ is cocommutative, then the sets described
in ii above are bases of eigenvectors for the appropriate $\orif\tautilde$
operator.
\end{enumerate}
\end{thm}

\begin{thm}
\label{thm:rotriffleodd-evector} Let $\calh$ be a graded connected
Hopf algebra, and $\tau:\calh\rightarrow\calh$ a linear graded involution
that is an algebra morphism and a coalgebra morphism or antimorphism.
Fix $a$ odd.

\begin{enumerate}
\item For $p_{1},\dots,p_{k}\in(Prim\calh)^{\tau}$ and $\barp_{1},\dots,\barp_{\bark}\in(Prim\calh)^{-\tau}$(where
$k$ or $\bark$ may be zero),
\begin{equation}
\sum_{B}\overleftarrow{\prod_{i\in B_{1}}}\barp_{i}\left(\sum_{\sigma\in\sk}p_{\sigma(1)}\dots p_{\sigma(k)}\right)\prod_{i\in B_{2}}\barp_{i},\label{eq:rotriffle-evector}
\end{equation}
summing over all set-compositions $B=B_{1}|B_{2}$ of $\{1,2,\dots,\bark\}$
into 2 blocks, is an eigenvector of $\oriftauplus$ with eigenvalue
$a^{k}$, and an eigenvector of $\oriftauminus$ with eigenvalue $(-1)^{\bark}a^{k}$.
($\overleftarrow{\prod}$ denotes that the product should be taken
in the reverse order, with large index $i$ on the left and small
index $i$ on the right - see Example \ref{ex:rotriffle-evector}
below.)
\item If $\calp,\barcalp$ are ordered bases of $(Prim\calh)^{\tau},(Prim\calh)^{-\tau}$
respectively, then the vectors in (\ref{eq:rotriffle-evector}), over
all choices of $p_{1}\preceq\dots\preceq p_{k}\in\calp$ and $\barp_{1}\preceq\dots\preceq\barp_{\bark}\in\barcalp$
(\textcolor{green}{allowing $k$ and $\bark$ to vary}), are linearly
independent.
\item Furthermore, if $\calh$ is cocommutative, then the set described
in ii above is a basis of eigenvectors for $\oriftauplus$ and $\oriftauminus$.
\end{enumerate}
\end{thm}

\begin{example}
\label{ex:rotriffle-evector} If $k=2$ and $\bark=3$, then the eigenvector
given by (\ref{eq:rotriffle-evector}) is
\begin{align*}
\barp_{3}\barp_{2}\barp_{1}(p_{1}p_{2}+p_{2}p_{1}) & +\barp_{2}\barp_{1}(p_{1}p_{2}+p_{2}p_{1})\barp_{3}\\
+\barp_{3}\barp_{1}(p_{1}p_{2}+p_{2}p_{1})\barp_{2} & +\barp_{3}\barp_{2}(p_{1}p_{2}+p_{2}p_{1})\barp_{1}\\
+\barp_{1}(p_{1}p_{2}+p_{2}p_{1})\barp_{2}\barp_{3} & +\barp_{2}(p_{1}p_{2}+p_{2}p_{1})\barp_{1}\barp_{3}\\
+\barp_{3}(p_{1}p_{2}+p_{2}p_{1})\barp_{1}\barp_{2} & +(p_{1}p_{2}+p_{2}p_{1})\barp_{1}\barp_{2}\barp_{3}.
\end{align*}
For a further example in the free associative algebra, see Example
\ref{ex:rotriffle-evectoralgo}.
\end{example}

When $a$ is even, we can only identify the eigenvectors of $\oriftauplus$
and $\oriftauminus$ with non-zero eigenvalues:
\begin{thm}
\label{thm:rotriffleeven-evector} Let $\calh$ be a graded connected
Hopf algebra, and $\tau:\calh\rightarrow\calh$ a linear graded involution
that is an algebra morphism and a coalgebra morphism or antimorphism.
Fix $a$ even.

\begin{enumerate}
\item For $p_{1},\dots,p_{k}\in(Prim\calh)^{\tau}$,
\[
\sum_{\sigma\in\sk}p_{\sigma(1)}\dots p_{\sigma(k)}
\]
 is an eigenvector of $\oriftauplus$ and $\oriftauminus$ with eigenvalue
$a^{k}$.
\item If $\calp$ is an ordered basis of $(Prim\calh)^{\tau}$, then the
vectors described in i above, over all choices of $p_{1}\preceq\dots\preceq p_{k}\in\calp$,
are linearly independent.
\item Furthermore, if $\calh$ is cocommutative, then the set of eigenvectors
described in ii above, for each fixed $k$, is a basis for the eigenspace
of eigenvalue $a^{k}$, for both $\oriftauplus$ and $\oriftauminus$,
and \textcolor{green}{there are no generalised eigenvectors for these
eigenvalues}. In other words, the eigenvectors described in ii, over
all values of $k$, \textcolor{green}{is a basis for a complement
of the generalised eigenspace of eigenvalue 0}. 
\end{enumerate}
\end{thm}

The proof of Theorems \ref{thm:flipriffle-evector} and \ref{thm:rotriffleodd-evector}
follow the same structure of an induction on $\bark$, the number
of $\tau$-negating (or $\tautilde$-negating) primitives; Lemma \ref{lem:symmetrisedproduct-evector}
below is the common base case ($\bark=0$), and\textcolor{green}{{}
to increase $\bark$ we use a sign-reversing involution argument:
Lemma \ref{lem:rotriffle-onenegatingprimitive-evector} for an algebra
morphism, and Lemma \ref{lem:flipriffle-onenegatingprimitive-evector}
for an algebra antimorphism.} Since Theorem \ref{thm:rotriffleeven-evector}
does not involve $\tau$-negating primitives, its proof requires only
the previous ``base case'' of Lemma \ref{lem:symmetrisedproduct-evector},
plus a \textcolor{green}{counting argument}. \textcolor{green}{Note
that this Lemma applies to all four types of Hopf ambimorphism, i.e.
to both $\tau$ and $\tautilde$.}
\begin{lem}
\label{lem:symmetrisedproduct-evector} Let $\calh$ be a graded connected
Hopf algebra, and $\tau:\calh\rightarrow\calh$ be a linear graded
involution that is a Hopf ambimorphism. For $p_{1},\dots,p_{k}\in(Prim\calh)^{\tau}$,
the symmetrised product $\sum_{\sigma\in\sk}p_{\sigma(1)}\dots p_{\sigma(k)}$
is an eigenvector of $\oriftauplus$ and $\oriftauminus$, of eigenvalue
$a^{k}$. 
\end{lem}

\begin{proof}
The argument is essentially the same for all four operators of (\ref{eq:oriffle-def}). For concreteness, consider first $\oriftautildeplus$
where $a$ is odd.
\[
\Delta^{[a]}(p_{1}\dots p_{k})=\sum_{B}\left(\prod_{i\in B_{1}}p_{i}\right)\otimes\left(\prod_{i\in B_{2}}p_{i}\right)\otimes\dots\otimes\left(\prod_{i\in B_{a}}p_{i}\right),
\]
summing over all set-compositions $B$ of $\{1,\dots,k\}$ with $a$
blocks. So
\begin{equation}
(\id\otimes\tautilde\otimes\id\otimes\tautilde\dots\otimes\id)\Delta^{[a]}(p_{1}\dots p_{k})=\sum_{B}\left(\prod_{i\in B_{1}}p_{i}\right)\otimes\left(\overleftarrow{\prod_{i\in B_{2}}}p_{i}\right)\otimes\dots\otimes\left(\prod_{i\in B_{a}}p_{i}\right),\label{eq:flipriffle-symlemma}
\end{equation}
where the arrows above the product signs indicate reversing the order
of the product. Hence $\oriftautildeplus(p_{1}\dots p_{k})$ is a
sum of products of the same primitives, in a different order, and
the permissable orders do not depend on anything particular about
each $p_{i}$. Thus $\sum_{\sigma\in\sk}p_{\sigma(1)}\dots p_{\sigma(k)}$
is an eigenvector of $\oriftautildeplus$ and the eigenvalue is the
number of summands in (\ref{eq:flipriffle-symlemma}), i.e. the number
of set-compositions of $\{1,\dots,k\}$ with $a$ blocks. Since each
number can belong to any block, the required number is $a^{k}$.

The other seven cases ($\oriftauplus$ with $a$ even, $\oriftauminus$
with $a$ even or $a$ odd, $\oriftautildeplus$ and $\oriftautildeminus$
for all values of $a$) are similar: the directions of the products
in (\ref{eq:flipriffle-symlemma}) are different, but each summand
is nevertheless a product of the same primitives in a different order,
hence the argument above applies.
\end{proof}

\begin{proof}[Proof of Theorem \ref{thm:rotriffleeven-evector}]
 Let $\beta_{\lambda,\barlambda}^{a,+}$ (resp. $\beta_{\lambda,\barlambda}^{a,-}$)
denote the eigenvalue of $\oriftauplus$ (resp. $\oriftauminus$)
corresponding to the double-partition $\lambda,\barlambda$, as
in the proof of Proposition \ref{prop:riffle-composition}.
\begin{enumerate}
\item Immediate from Lemma \ref{lem:symmetrisedproduct-evector}.
\item Work in the subalgebra of $\calh$ that is the universal enveloping
algebra of $(Prim\calh)^{\tau}$, with a PBW basis formed from $\calp$.
As explained in the Triangularity Lemma (\ref{lem:trilemma}), the
eigenvector $\sum_{\sigma\in\sk}p_{\sigma(1)}\dots p_{\sigma(k)}$
is a sum of products of $p_{1},\dots,p_{k}$ in different orders,
and hence its highest length term in this PBW basis is $k!p_{1}\dots p_{k}$.
Thus each eigenvector has a different highest length term, and so
they are linearly independent.
\item Let $\barcalp$ be a basis of $(Prim\calh)^{-\tau}$. As noted in
\textcolor{green}{the paragraph after Lemma \ref{lem:trilemma}}:
in a cocommutative Hopf algebra, the multiplicity of the eigenvalue
$\beta_{\lambda,\barlambda}^{a,+}$ or $\beta_{\lambda,\barlambda}^{a,-}$
is the number of multiset pairs $\{p_{1},\dots,p_{k}\}\subseteq\calp$,
$\{\barp_{1}\dots,\barp_{\bark}\}\subseteq\barcalp$ where $\lambda=(\deg p_{1},\dots,\deg p_{k})$
and $\barlambda=(\deg\barp_{1},\dots,\deg\barp_{\bark})$. \textcolor{green}{The
constructed eigenvectors accounts for all cases where $\barlambda=\emptyset$.
And, in the }proof of Proposition \ref{prop:riffle-composition},
a sign-reversing involution showed that $\beta_{\lambda,\barlambda}^{a,+}=\beta_{\lambda,\barlambda}^{a,-}=0$
if $\barlambda\neq\emptyset$.
\end{enumerate}
\end{proof}
The following two lemmas are the sign-reversing involutions required
to inductively construct eigenvectors.
\begin{lem}
\label{lem:flipriffle-onenegatingprimitive-evector} Let $\calh$
be a graded connected Hopf algebra, and $\tautilde:\calh\rightarrow\calh$
be a linear graded involution that is an algebra antimorphism and
a coalgebra morphism or antimorphism. Let $\barp$ be a $\tautilde$-negating
primitive element of $\calh$, and $s$ be any product of primitive
elements. Then 
\begin{alignat}{3}
\mbox{for }a\mbox{ odd}: & \quad & \oriftautildeplus(\barp s) & =\barp\oriftautildeplus(s), & \oriftautildeplus(s\barp) & =\oriftautildeplus(s)\barp,\nonumber \\
 &  & \oriftautildeminus(\barp s) & =-\oriftautildeminus(s)\barp,\quad & \oriftautildeminus(s\barp) & =-\barp\oriftautildeminus(s);\nonumber \\
\mbox{for }a\mbox{ even}: &  &  &  & \oriftautildeplus(s\barp) & =0,\nonumber \\
 &  & \oriftautildeminus(\barp s) & =0.\label{eq:flipriffle-onenegatingprimitive-evector}
\end{alignat}
\end{lem}

\begin{proof}
Consider first the equations on the left, involving $\barp s$. Let
$s=p_{1}\dots p_{k}$ and, for clearer exposition, first set $a=5$.
\begin{align}
\Delta^{[5]}(\barp p_{1}\dots p_{k}) & =\sum_{B}\left(\barp\prod_{i\in B_{1}}p_{i}\right)\otimes\left(\prod_{i\in B_{2}}p_{i}\right)\otimes\left(\prod_{i\in B_{3}}p_{i}\right)\otimes\left(\prod_{i\in B_{4}}p_{i}\right)\otimes\left(\prod_{i\in B_{5}}p_{i}\right)\nonumber \\
 & \qquad\qquad+\left(\prod_{i\in B_{1}}p_{i}\right)\otimes\left(\barp\prod_{i\in B_{2}}p_{i}\right)\otimes\left(\prod_{i\in B_{3}}p_{i}\right)\otimes\left(\prod_{i\in B_{4}}p_{i}\right)\otimes\left(\prod_{i\in B_{5}}p_{i}\right)\nonumber \\
 & \qquad\qquad+\left(\prod_{i\in B_{1}}p_{i}\right)\otimes\left(\prod_{i\in B_{2}}p_{i}\right)\otimes\left(\barp\prod_{i\in B_{3}}p_{i}\right)\otimes\left(\prod_{i\in B_{4}}p_{i}\right)\otimes\left(\prod_{i\in B_{5}}p_{i}\right)\nonumber \\
 & \qquad\qquad+\left(\prod_{i\in B_{1}}p_{i}\right)\otimes\left(\prod_{i\in B_{2}}p_{i}\right)\otimes\left(\prod_{i\in B_{3}}p_{i}\right)\otimes\left(\barp\prod_{i\in B_{4}}p_{i}\right)\otimes\left(\prod_{i\in B_{5}}p_{i}\right)\nonumber \\
 & \qquad\qquad+\left(\prod_{i\in B_{1}}p_{i}\right)\otimes\left(\prod_{i\in B_{2}}p_{i}\right)\otimes\left(\prod_{i\in B_{3}}p_{i}\right)\otimes\left(\prod_{i\in B_{4}}p_{i}\right)\otimes\left(\barp\prod_{i\in B_{5}}p_{i}\right),\label{eq:fivefold-coproduct-primitives}
\end{align}
summing over all set-compositions $B$ of $\{1,2,\dots,k\}$ into
5 parts. (Each such set-composition contributed 5 terms, \textcolor{green}{for
all 5 possible blocks where $\barp$ may be assigned.}) So
\begin{align}
\orif\tautilde_{5}^{+}(\barp p_{1}\dots p_{k}) & =\sum_{B}\left(\barp\prod_{i\in B_{1}}p_{i}\right)\tautilde\left(\prod_{i\in B_{2}}p_{i}\right)\left(\prod_{i\in B_{3}}p_{i}\right)\tautilde\left(\prod_{i\in B_{4}}p_{i}\right)\left(\prod_{i\in B_{5}}p_{i}\right)\nonumber \\
 & \qquad\qquad-\left(\prod_{i\in B_{1}}p_{i}\right)\left[\tautilde\left(\prod_{i\in B_{2}}p_{i}\right)\barp\right]\left(\prod_{i\in B_{3}}p_{i}\right)\tautilde\left(\prod_{i\in B_{4}}p_{i}\right)\left(\prod_{i\in B_{5}}p_{i}\right)\nonumber \\
 & \qquad\qquad+\left(\prod_{i\in B_{1}}p_{i}\right)\tautilde\left(\prod_{i\in B_{2}}p_{i}\right)\left(\barp\prod_{i\in B_{3}}p_{i}\right)\tautilde\left(\prod_{i\in B_{4}}p_{i}\right)\left(\prod_{i\in B_{5}}p_{i}\right)\nonumber \\
 & \qquad\qquad-\left(\prod_{i\in B_{1}}p_{i}\right)\tautilde\left(\prod_{i\in B_{2}}p_{i}\right)\left(\prod_{i\in B_{3}}p_{i}\right)\left[\tautilde\left(\prod_{i\in B_{4}}p_{i}\right)\barp\right]\left(\prod_{i\in B_{5}}p_{i}\right)\nonumber \\
 & \qquad\qquad+\left(\prod_{i\in B_{1}}p_{i}\right)\tautilde\left(\prod_{i\in B_{2}}p_{i}\right)\left(\prod_{i\in B_{3}}p_{i}\right)\tautilde\left(\prod_{i\in B_{4}}p_{i}\right)\left(\barp\prod_{i\in B_{5}}p_{i}\right).\label{eq:flipriffle-onenegativeprimitive-1}
\end{align}
(The second factor of the second line and the fourth factor of the
fourth line uses, for all $x$, that $\tautilde(\barp x)=\tautilde(x)\tautilde(\barp)=-\tautilde(x)\barp$.)
Note that, for each fixed set-composition $B$, the second and third
lines of (\ref{eq:flipriffle-onenegativeprimitive-1}) are equal,
except for their opposite sign, so they cancel. Similarly, the fourth
and fifth lines cancel. Thus only the first line remains, and that
is precisely $\barp\orif\tautilde_{5}^{+}(p_{1}\dots p_{k})$.

By this argument, whenever $a$ is odd, $\oriftautildeplus(\barp p_{1}\dots p_{k})$
is a sum of $a$ terms for each set-composition of $\{1,2,\dots,k\}$
into $a$ parts, and each even term cancels with the following odd
term, so only the first term remains and these give $\barp\oriftautildeplus(p_{1}\dots p_{k})$.

Similarly, 
\begin{align*}
\oriftautildeminus(\barp p_{1}\dots p_{k}) & =\sum_{B}-\left[\tautilde\left(\prod_{i\in B_{1}}p_{i}\right)\barp\right]\left(\prod_{i\in B_{2}}p_{i}\right)\tautilde\left(\prod_{i\in B_{3}}p_{i}\right)\dots\\
 & \qquad\qquad+\tautilde\left(\prod_{i\in B_{1}}p_{i}\right)\left(\barp\prod_{i\in B_{2}}p_{i}\right)\tautilde\left(\prod_{i\in B_{3}}p_{i}\right)\dots\\
 & \qquad\qquad-\tautilde\left(\prod_{i\in B_{1}}p_{i}\right)\left(\prod_{i\in B_{2}}p_{i}\right)\left[\tautilde\left(\prod_{i\in B_{3}}p_{i}\right)\barp\right]\dots\\
 & \qquad\qquad+\dots,
\end{align*}
so each odd line cancels with the line below. Hence, if $a$ is even,
the sum entirely cancels and is thus 0; if $a$ is odd, the last line
remains: 
\[
\oriftautildeminus(\barp p_{1}\dots p_{k})=\sum_{B}-\tautilde\left(\prod_{i\in B_{1}}p_{i}\right)\left(\prod_{i\in B_{2}}p_{i}\right)\tautilde\left(\prod_{i\in B_{3}}p_{i}\right)\dots\left[\tautilde\left(\prod_{i\in B_{a}}p_{i}\right)\barp\right]=\oriftautildeminus(p_{1}\dots p_{k})\barp.
\]

For $s\barp$, the calculations are very similar. Again, let $s=p_{1}\dots p_{k}$.
\begin{align*}
\oriftautildeplus(p_{1}\dots p_{k}\barp) & =\sum_{B}\left[\left(\prod_{i\in B_{1}}p_{i}\right)\barp\right]\tautilde\left(\prod_{i\in B_{2}}p_{i}\right)\left(\prod_{i\in B_{3}}p_{i}\right)\dots\\
 & \qquad\qquad-\left(\prod_{i\in B_{1}}p_{i}\right)\left[\barp\tautilde\left(\prod_{i\in B_{2}}p_{i}\right)\right]\left(\prod_{i\in B_{3}}p_{i}\right)\dots\\
 & \qquad\qquad+\left(\prod_{i\in B_{1}}p_{i}\right)\tautilde\left(\prod_{i\in B_{2}}p_{i}\right)\left[\left(\prod_{i\in B_{3}}p_{i}\right)\barp\right]\dots\\
 & \qquad\qquad-\dots,
\end{align*}
so each odd line cancels with the line below, leaving the last line
(i.e. $\oriftautildeplus(p_{1}\dots p_{k})\barp$) if $a$ is odd,
and cancelling entirely if $a$ is even. And, for $a$ odd, 
\begin{align*}
\oriftautildeminus(p_{1}\dots p_{k}\barp) & =\sum_{B}-\left[\barp\tautilde\left(\prod_{i\in B_{1}}p_{i}\right)\right]\left(\prod_{i\in B_{2}}p_{i}\right)\tautilde\left(\prod_{i\in B_{3}}p_{i}\right)\dots\tautilde\left(\prod_{i\in B_{a}}p_{i}\right)\\
 & \qquad\qquad+\tautilde\left(\prod_{i\in B_{1}}p_{i}\right)\left[\left(\prod_{i\in B_{2}}p_{i}\right)\barp\right]\tautilde\left(\prod_{i\in B_{3}}p_{i}\right)\dots\left(\prod_{i\in B_{a}}p_{i}\right)\\
 & \qquad\qquad-\tautilde\left(\prod_{i\in B_{1}}p_{i}\right)\left(\prod_{i\in B_{2}}p_{i}\right)\left[\barp\tautilde\left(\prod_{i\in B_{3}}p_{i}\right)\right]\dots\tautilde\left(\prod_{i\in B_{a}}p_{i}\right)\\
 & \qquad\qquad+\dots\\
 & \qquad\qquad-\tautilde\left(\prod_{i\in B_{1}}p_{i}\right)\left(\prod_{i\in B_{2}}p_{i}\right)\tautilde\left(\prod_{i\in B_{3}}p_{i}\right)\dots\left[\barp\tautilde\left(\prod_{i\in B_{a}}p_{i}\right)\right].
\end{align*}
where each even line cancels with the line below, leaving the first
line (i.e. $\barp\oriftautildeminus(p_{1}\dots p_{k})$).
\end{proof}

\begin{lem}
\label{lem:rotriffle-onenegatingprimitive-evector} Let $\calh$ be
a graded connected Hopf algebra, and $\tau:\calh\rightarrow\calh$
be a linear graded involution that is an algebra morphism and a coalgebra
morphism or antimorphism. Let $\barp$ be a $\tau$-negating primitive
element of $\calh$, and $s$ be any product of primitive elements.
Then, if $a$ is odd: 
\begin{align}
\oriftauplus(\barp s+s\barp) & =\barp\oriftauplus(s)+\oriftauplus(s)\barp;\nonumber \\
\oriftauminus(\barp s+s\barp) & =-\barp\oriftauminus(s)-\oriftauminus(s)\barp.\label{eq:rotriffle-onenegatingprimitive-evector}
\end{align}
\end{lem}

\begin{proof}
Let $s=p_{1}\dots p_{k}$, and, as in the proof of Lemma \ref{lem:flipriffle-onenegatingprimitive-evector},
first set $a=5$ for clearer exposition. From (\ref{eq:fivefold-coproduct-primitives}),
we have
\begin{align}
\oriftauplus(\barp p_{1}\dots p_{k}) & =\sum_{B}\left(\barp\prod_{i\in B_{1}}p_{i}\right)\tau\left(\prod_{i\in B_{2}}p_{i}\right)\left(\prod_{i\in B_{3}}p_{i}\right)\tau\left(\prod_{i\in B_{4}}p_{i}\right)\left(\prod_{i\in B_{5}}p_{i}\right)\nonumber \\
 & \qquad\qquad-\left(\prod_{i\in B_{1}}p_{i}\right)\left[\barp\tau\left(\prod_{i\in B_{2}}p_{i}\right)\right]\left(\prod_{i\in B_{3}}p_{i}\right)\tau\left(\prod_{i\in B_{4}}p_{i}\right)\left(\prod_{i\in B_{5}}p_{i}\right)\nonumber \\
 & \qquad\qquad+\left(\prod_{i\in B_{1}}p_{i}\right)\tau\left(\prod_{i\in B_{2}}p_{i}\right)\left(\barp\prod_{i\in B_{3}}p_{i}\right)\tau\left(\prod_{i\in B_{4}}p_{i}\right)\left(\prod_{i\in B_{5}}p_{i}\right)\nonumber \\
 & \qquad\qquad-\left(\prod_{i\in B_{1}}p_{i}\right)\tau\left(\prod_{i\in B_{2}}p_{i}\right)\left(\prod_{i\in B_{3}}p_{i}\right)\left[\barp\tau\left(\prod_{i\in B_{4}}p_{i}\right)\right]\left(\prod_{i\in B_{5}}p_{i}\right)\nonumber \\
 & \qquad\qquad+\left(\prod_{i\in B_{1}}p_{i}\right)\tau\left(\prod_{i\in B_{2}}p_{i}\right)\left(\prod_{i\in B_{3}}p_{i}\right)\tau\left(\prod_{i\in B_{4}}p_{i}\right)\left(\barp\prod_{i\in B_{5}}p_{i}\right),\label{eq:rotriffle-onenegativeprimitive-1}
\end{align}
and similarly
\begin{align}
\oriftauplus(p_{1}\dots p_{k}\barp) & =\sum_{B}\left[\left(\prod_{i\in B_{1}}p_{i}\right)\barp\right]\tau\left(\prod_{i\in B_{2}}p_{i}\right)\left(\prod_{i\in B_{3}}p_{i}\right)\tau\left(\prod_{i\in B_{4}}p_{i}\right)\left(\prod_{i\in B_{5}}p_{i}\right)\nonumber \\
 & \qquad\qquad-\left(\prod_{i\in B_{1}}p_{i}\right)\left[\tau\left(\prod_{i\in B_{2}}p_{i}\right)\barp\right]\left(\prod_{i\in B_{3}}p_{i}\right)\tau\left(\prod_{i\in B_{4}}p_{i}\right)\left(\prod_{i\in B_{5}}p_{i}\right)\nonumber \\
 & \qquad\qquad+\left(\prod_{i\in B_{1}}p_{i}\right)\tau\left(\prod_{i\in B_{2}}p_{i}\right)\left[\left(\prod_{i\in B_{3}}p_{i}\right)\barp\right]\tau\left(\prod_{i\in B_{4}}p_{i}\right)\left(\prod_{i\in B_{5}}p_{i}\right)\nonumber \\
 & \qquad\qquad-\left(\prod_{i\in B_{1}}p_{i}\right)\tau\left(\prod_{i\in B_{2}}p_{i}\right)\left(\prod_{i\in B_{3}}p_{i}\right)\left[\tau\left(\prod_{i\in B_{4}}p_{i}\right)\barp\right]\left(\prod_{i\in B_{5}}p_{i}\right)\nonumber \\
 & \qquad\qquad+\left(\prod_{i\in B_{1}}p_{i}\right)\tau\left(\prod_{i\in B_{2}}p_{i}\right)\left(\prod_{i\in B_{3}}p_{i}\right)\tau\left(\prod_{i\in B_{4}}p_{i}\right)\left[\left(\prod_{i\in B_{5}}p_{i}\right)\barp\right].\label{eq:rotriffle-onenegativeprimitive-2}
\end{align}
Note that, for each fixed set-composition $B$ of $\{1,2,\dots,k\}$
into 5 parts, the second line of (\ref{eq:rotriffle-onenegativeprimitive-1})
is equal to the first line of (\ref{eq:rotriffle-onenegativeprimitive-2})
except for the opposite sign. So these will cancel in $\oriftauplus(\barp p_{1}\dots p_{k})+\oriftauplus(p_{1}\dots p_{k}\barp)$,
and similarly the third line of (\ref{eq:rotriffle-onenegativeprimitive-1})
and the second line of (\ref{eq:rotriffle-onenegativeprimitive-2})
will cancel, and the fourth line of (\ref{eq:rotriffle-onenegativeprimitive-1})
and the third line of (\ref{eq:rotriffle-onenegativeprimitive-2})
will cancel, and so on. Hence $\oriftauplus(\barp p_{1}\dots p_{k})+\oriftauplus(p_{1}\dots p_{k}\barp)$
is equal to the first line of (\ref{eq:rotriffle-onenegativeprimitive-1})
plus the last line of (\ref{eq:rotriffle-onenegativeprimitive-2})
- these are equal respectively to $\barp\oriftauplus(p_{1}\dots p_{k})$
and $\oriftauplus(p_{1}\dots p_{k})\barp$. The same cancellations
occur for other odd values of $a$.

As for $\oriftauminus$ (for $a$ odd):
\begin{align}
\oriftauminus(\barp p_{1}\dots p_{k}) & =\sum_{B}-\left[\barp\tau\left(\prod_{i\in B_{1}}p_{i}\right)\right]\left(\prod_{i\in B_{2}}p_{i}\right)\tau\left(\prod_{i\in B_{3}}p_{i}\right)\dots\tau\left(\prod_{i\in B_{a}}p_{i}\right)\nonumber \\
 & \qquad\qquad+\tau\left(\prod_{i\in B_{1}}p_{i}\right)\left(\barp\prod_{i\in B_{2}}p_{i}\right)\tau\left(\prod_{i\in B_{3}}p_{i}\right)\dots\tau\left(\prod_{i\in B_{a}}p_{i}\right)\nonumber \\
 & \qquad\qquad-\tau\left(\prod_{i\in B_{1}}p_{i}\right)\left(\prod_{i\in B_{2}}p_{i}\right)\left[\barp\tau\left(\prod_{i\in B_{3}}p_{i}\right)\right]\dots\tau\left(\prod_{i\in B_{a}}p_{i}\right)\nonumber \\
 & \qquad\qquad+\dots\nonumber \\
 & \qquad\qquad-\tau\left(\prod_{i\in B_{1}}p_{i}\right)\left(\prod_{i\in B_{2}}p_{i}\right)\tau\left(\prod_{i\in B_{3}}p_{i}\right)\dots\left[\barp\tau\left(\prod_{i\in B_{a}}p_{i}\right)\right],\label{eq:rotriffle-onenegativeprimitive-3}
\end{align}
and 
\begin{align}
\oriftauminus(p_{1}\dots p_{k}\barp) & =\sum_{B}-\left[\tau\left(\prod_{i\in B_{1}}p_{i}\right)\barp\right]\left(\prod_{i\in B_{2}}p_{i}\right)\tau\left(\prod_{i\in B_{3}}p_{i}\right)\dots\tau\left(\prod_{i\in B_{a}}p_{i}\right)\nonumber \\
 & \qquad\qquad+\tau\left(\prod_{i\in B_{1}}p_{i}\right)\left[\left(\prod_{i\in B_{2}}p_{i}\right)\barp\right]\tau\left(\prod_{i\in B_{3}}p_{i}\right)\dots\tau\left(\prod_{i\in B_{a}}p_{i}\right)\nonumber \\
 & \qquad\qquad-\tau\left(\prod_{i\in B_{1}}p_{i}\right)\left(\prod_{i\in B_{2}}p_{i}\right)\left[\tau\left(\prod_{i\in B_{3}}p_{i}\right)\barp\right]\dots\tau\left(\prod_{i\in B_{a}}p_{i}\right)\nonumber \\
 & \qquad\qquad+\dots\nonumber \\
 & \qquad\qquad-\tau\left(\prod_{i\in B_{1}}p_{i}\right)\left(\prod_{i\in B_{2}}p_{i}\right)\tau\left(\prod_{i\in B_{3}}p_{i}\right)\dots\left[\tau\left(\prod_{i\in B_{a}}p_{i}\right)\barp\right].\label{eq:rotriffle-onenegativeprimitive-4}
\end{align}
So, when (\ref{eq:rotriffle-onenegativeprimitive-3}) and (\ref{eq:rotriffle-onenegativeprimitive-4})
are summed, again the $i$th line of (\ref{eq:rotriffle-onenegativeprimitive-3})
cancels with the $i+1$th line of (\ref{eq:rotriffle-onenegativeprimitive-4}),
leaving the first line of (\ref{eq:rotriffle-onenegativeprimitive-3})
and the last line of (\ref{eq:rotriffle-onenegativeprimitive-4}).
\end{proof}
\textcolor{green}{With these sign-reversing involutions in place,
}we proceed to prove the eigenvector formulas.

\begin{proof}[Proof of Theorem \ref{thm:flipriffle-evector}]

\begin{enumerate}
\item First note that all lines in (\ref{eq:flipriffle-onenegatingprimitive-evector})
are linear in $s$, i.e. Lemma \ref{lem:flipriffle-onenegatingprimitive-evector}
holds when $s$ is a linear combination of products of primitives.

\begin{enumerate}
\item The $\bark=0$ case is Lemma \ref{lem:symmetrisedproduct-evector}.
For $\bark>0$, use the third and fourth lines of (\ref{eq:flipriffle-onenegatingprimitive-evector}),
with $s=\left(\sum_{\sigma\in\sk}p_{\sigma(1)}\dots p_{\sigma(k)}\right)\barp_{1}\dots\barp_{\bark-1}$
and $s=\barp_{2}\dots\barp_{\bark}\left(\sum_{\sigma\in\sk}p_{\sigma(1)}\dots p_{\sigma(k)}\right)$
respectively.
\item Proceed by induction on $\bark$, the base case $\bark=0$ being Lemma
\ref{lem:symmetrisedproduct-evector}. By the first line of (\ref{eq:flipriffle-onenegatingprimitive-evector}):
\begin{align*}
\oriftautildeplus\left(\barp_{1}\dots\barp_{\bark}\sum_{\sigma\in\sk}p_{\sigma(1)}\dots p_{\sigma(k)}\right) & =\barp_{1}\oriftautildeplus\left(\barp_{2}\dots\barp_{\bark}\sum_{\sigma\in\sk}p_{\sigma(1)}\dots p_{\sigma(k)}\right)\\
 & =\barp_{1}a^{k}\barp_{2}\dots\barp_{\bark}\sum_{\sigma\in\sk}p_{\sigma(1)}\dots p_{\sigma(k)};
\end{align*}
and 
\begin{align*}
\oriftautildeplus\left[\left(\sum_{\sigma\in\sk}p_{\sigma(1)}\dots p_{\sigma(k)}\right)\barp_{1}\dots\barp_{\bark}\right] & =\oriftautildeplus\left[\left(\sum_{\sigma\in\sk}p_{\sigma(1)}\dots p_{\sigma(k)}\right)\barp_{1}\dots\barp_{\bark-1}\right]\barp_{\bark}\\
 & =a^{k}\left[\left(\sum_{\sigma\in\sk}p_{\sigma(1)}\dots p_{\sigma(k)}\right)\barp_{1}\dots\barp_{\bark-1}\right]\barp_{\bark}
\end{align*}
using the inductive hypothesis at the second equality each time.
\item Repeatedly applying the second line of (\ref{eq:flipriffle-onenegatingprimitive-evector}):
\begin{align*}
\oriftautildeminus\left(\barp_{1}\barp_{2}\dots\barp_{\bark}\sum_{\sigma\in\sk}p_{\sigma(1)}\dots p_{\sigma(k)}\right) & =-\oriftautildeminus\left(\barp_{2}\dots\barp_{\bark}\sum_{\sigma\in\sk}p_{\sigma(1)}\dots p_{\sigma(k)}\right)\barp_{1}\\
 & =\oriftautildeminus\left(\barp_{3}\dots\barp_{\bark}\sum_{\sigma\in\sk}p_{\sigma(1)}\dots p_{\sigma(k)}\right)\barp_{2}\barp_{1}\\
 & \vdots\\
 & =(-1)^{\bark}\oriftautildeminus\left(\sum_{\sigma\in\sk}p_{\sigma(1)}\dots p_{\sigma(k)}\right)\barp_{k}\dots\barp_{2}\barp_{1}\\
 & =(-1)^{\bark}a^{k}\left(\sum_{\sigma\in\sk}p_{\sigma(1)}\dots p_{\sigma(k)}\right)\barp_{k}\dots\barp_{2}\barp_{1},
\end{align*}
and by the same recursive process 
\[
\oriftautildeminus\left[\left(\sum_{\sigma\in\sk}p_{\sigma(1)}\dots p_{\sigma(k)}\right)\barp_{k}\dots\barp_{2}\barp_{1}\right]=(-1)^{\bark}a^{k}\barp_{1}\barp_{2}\dots\barp_{\bark}\left(\sum_{\sigma\in\sk}p_{\sigma(1)}\dots p_{\sigma(k)}\right).
\]
Summing these gives the required eigenvector.
\end{enumerate}
\item Work in the subalgebra of $\calh$ that is the universal enveloping
algebra of $Prim(\calh)$. As in the Triangularity Lemma (\ref{lem:trilemma}),
consider its PBW basis formed from the basis $\calp\cup\barcalp$
of $Prim(\calh)$, with the concatentation order on $\calp\cup\barcalp$.
Then, for each \textcolor{green}{fixed format} in part i of the Theorem,
the eigenvector made from $p_{1},\dots,p_{k}\in\calp$ and $\barp_{1},\dots,\barp_{\bark}\in\barcalp$
is a sum of their products in different orders, and hence its highest
length term in this PBW basis is $p_{1}\dots p_{k}\barp_{1}\dots\barp_{\bark}$
(its coefficient is the \textcolor{green}{number of product terms}
in the eigenvector, and is hence non-zero). Thus each eigenvector has
a different highest length term, and so they are linearly independent.
\item As explained above, each set of eigenvectors is triangular with respect
to the PBW basis of the universal enveloping algebra of $Prim(\calh)$,
thus giving a basis of this universal enveloping algebra. By the Cartier-Milnor-Moore
theorem, when $\calh$ is cocommutative, this universal enveloping
algebra is precisely $\calh$.
\end{enumerate}
\end{proof}

\begin{proof}[Proof of Theorem \ref{thm:rotriffleodd-evector}]

\begin{enumerate}
\item Proceed by induction on $\bark$, the base case of $\bark=0$ being
Lemma \ref{lem:symmetrisedproduct-evector}. As noted in the previous
proof, all lines in (\ref{eq:flipriffle-onenegatingprimitive-evector})
are linear in $s$, so we may apply Lemma \ref{lem:flipriffle-onenegatingprimitive-evector}
when $s$ is the eigenvector in (\ref{eq:rotriffle-evector}):
\begin{align*}
\oriftauplus(\barp s+s\barp) & =\barp\oriftauplus(s)+\oriftauplus(s)\barp=\barp(a^{k}s)+(a^{k}s)\barp=a^{k}(\barp s+s\barp);\\
\oriftauminus(\barp s+s\barp) & =-\barp\oriftauminus(s)-\oriftauminus(s)\barp=-\barp(-1)^{\bark}a^{k}s-(-1)^{\bark}a^{k}s\barp\\
 & \phantom{=-\barp\oriftauminus(s)-\oriftauminus(s)\barp}=(-1)^{\bark+1}a^{k}(\barp s+s\barp).
\end{align*}
So $\barp s+s\barp$ is an eigenvector with the required eigenvalue.
It remains to show that, when $s$ is of the form (\ref{eq:rotriffle-evector}),
then so is $\barp s+s\barp$, with $\bark$ increased by 1. To do
so, let $\barp=\barp_{\bark+1}$. Then, for each set-composition $B$
of $\{1,2,\dots,\bark\}$ into 2 blocks, 
\begin{align*}
 & \barp_{\bark+1}\left(\overleftarrow{\prod_{i\in B_{1}}}\barp_{i}\left(\sum_{\sigma\in\sk}p_{\sigma(1)}\dots p_{\sigma(k)}\right)\prod_{i\in B_{2}}\barp_{i}\right)+\left(\overleftarrow{\prod_{i\in B_{1}}}\barp_{i}\left(\sum_{\sigma\in\sk}p_{\sigma(1)}\dots p_{\sigma(k)}\right)\prod_{i\in B_{2}}\barp_{i}\right)\barp_{\bark+1}\\
= & \overleftarrow{\prod_{i\in B'_{1}}}\barp_{i}\left(\sum_{\sigma\in\sk}p_{\sigma(1)}\dots p_{\sigma(k)}\right)\prod_{i\in B'_{2}}\barp_{i}+\overleftarrow{\prod_{i\in B''_{1}}}\barp_{i}\left(\sum_{\sigma\in\sk}p_{\sigma(1)}\dots p_{\sigma(k)}\right)\prod_{i\in B''_{2}}\barp_{i},
\end{align*}
where $B'$ and $B''$ are set-compositions of $\{1,2,\dots,\bark+1\}$,
obtained from $B$ respectively by adding $\bark+1$ to $B_{1}$ or
to $B_{2}$. And all set-compositions of $\{1,2,\dots,\bark+1\}$ into 2 blocks
arise from a unique such $B$ in this way.
\item[ii, iii.] The same argument as in the proof of Theorem \ref{thm:flipriffle-evector}
above.
\end{enumerate}
\end{proof}

\section{Markov Chains from Hyperoctahedral Descent Operators\label{sec:chain}}

One application of the eigenvalues and eigenvectors of hyperoctahedral
descent operators calculated in previous sections is to study an associated
Markov chain, generalising the type A framework in \cite{hopfpowerchains,descentoperatorchains}.
Each positive hyperoctahedral descent operator, applied to each basis
of a Hopf algebra, determines a different Markov chain. We will be
minimal here, and refer the reader to \cite{descentoperatorchains}
and \cite{markovmixing} for more background on Markov chains.

\subsection{Chain Construction\label{subsec:Chain-Construction}}

Given a finite set $\calb$ and a matrix $K$ with rows and columns
labelled by elements of $\calb$, the \emph{Markov chain} with \emph{state
space} $\Omega$ and \emph{transition matrix} $K$ is a sequence of
random variables $X_{1},X_{2},\dots$ taking values in $\calb$,
such that the conditional probability
\[
\Prob(X_{t}=y|X_{t-1}=x,X_{t-2}=x_{t-2},\dots,X_{1}=x_{1})=\Prob(X_{t}=y|X_{t-1}=x)=K(x,y).
\]
In other words, $X_{t}$, the state at time $t$, is only dependent
on the state one timestep prior, not on further past history.

We would like this transition matrix $K$ to be (the transpose of)
the matrix of a hyperoctahedral descent operator, acting on a graded
Hopf algebra $\calh$, relative to some fixed basis $\calb$ (up to
scaling). To be probabilities, the matrix entries must be non-negative;
this motivates condition iv below, that the product and coproduct
of basis elements expand positively in $\calb$. Condition v, the
positivity of $\eta$, is required to scale the matrix so its rows
sum to 1, as probabilities should. Also, in condition ii, we require
the involutive Hopf ambimorphism $\tau$ to send the basis $\calb$
to itself (as opposed to to a linear combination of basis elements),
so that it can be interpreted as an involution on the combinatorial
objects indexing $\calb$, analogous to flipping a deck of cards.

\begin{setup}\label{setup:chain}The following conditions and notations
will be assumed when analysing Markov chains driven by hyperoctahedral
descent operators:
\begin{enumerate}
\item $\calh=\bigoplus_{n\geq0}\calhn$ is a graded connected Hopf algebra
over $\mathbb{R}$, and each $\calhn$ is finite-dimensional, with
basis $\calb_{n}$. And $\calb$ denotes $\amalg_{n\geq0}\calbn$.
\item $\tau:\calb\rightarrow\calb$ is a graded involution that extends
linearly to a Hopf ambimorphism on $\calh$.
\item $n$ is a fixed integer, and $P$ is a probability distribution on
signed weak-compositions of $n$ or on tilde-signed weak-compositions
of $n$ (but not a mixture of both). Define 
\begin{equation}
m\circ\Delta_{P(\tau)}:=\sum_{D}\frac{P(D)}{\binom{n}{D^{+}}}m\circ\Delta_{D(\tau)}\label{eq:mdeltaP}
\end{equation}
where $\binom{n}{D^{+}}$ is the multinomial coefficient $\binom{n}{|d_{1}|\:\dots\:|d_{l(D)}|}$.
\item For each signed or tilde-signed weak-composition $D$ with non-zero
probability under $P$, \textcolor{green}{and for all $x,y\in\calbn$,
all $z_{1}\in\calb_{|d_{1}|},z_{2}\in\calb_{|d_{2}|},\dots,z_{l(D)}\in\calb_{|d_{l(D)}|}$}:
\begin{enumerate}
\item $z_{1}z_{2}\dots z_{l(D)}=\sum_{y\in\calb_{n}}\xi_{z_{1},\dots,z_{l(D)}}^{y}y$
with $\xi_{z_{1},\dots,z_{l(D)}}^{y}\geq0$;
\item $\Delta_{D^{+}}(x)=\sum_{z_{i}\in\calb_{|d_{i}|}}\eta_{x}^{z_{1},\dots,z_{l(D)}}z_{1}\otimes z_{2}\otimes\dots\otimes z_{l(D)}$
with $\eta_{x}^{z_{1},\dots,z_{l(D)}}\geq0$.
\end{enumerate}
\item For all $x\in\calb_{n}$, the function 
\[
\eta(x):=\mbox{sum of coefficients (in the }\calb_{1}\otimes\dots\otimes\calb_{1}\mbox{ basis) of }\Delta_{1,\dots,1}(x)
\]
evaluates to a positive number.
\end{enumerate}
\end{setup}
\begin{thm}
\textcolor{green}{\label{thm:chain-3step} }Under Setup \ref{setup:chain},
the matrix
\begin{equation}
K(x,y):=\frac{\eta(y)}{\eta(x)}\mbox{ coefficient of }y\mbox{ in }m\circ\Delta_{P(\tau)}(x)\label{eq:transitionmtx}
\end{equation}
is a transition matrix, and each step of the associated Markov chain,
starting at $x\in\calbn$, is equivalent to the following four-step
process:

\begin{enumerate}[label=\arabic*.]
\item Choose a signed or tilde-signed weak-composition $D$ according to
the distribution $P$.
\item Choose $z_{1}\in\calb_{|d_{1}|},z_{2}\in\calb_{|d_{2}|},\dots,z_{l(D)}\in\calb_{|d_{l(D)}|}$
with probability $\frac{1}{\eta(x)}\eta_{x}^{z_{1},\dots,z_{l(D)}}\eta(z_{1})\dots\eta(z_{l(D)})$.
\item For all $i$ such that $d_{i}$ is a decorated part, replace $z_{i}$
by $\tau(z_{i})$.
\item Choose $y\in\calbn$ with probability $\left(\binom{n}{D}\eta(z_{1})\dots\eta(z_{l})\right)^{-1}\xi_{z'_{1},\dots,z'_{l}}^{y}\eta(y)$,
where $z_{i}'=z_{i}$ if $d_{i}$ is an undecorated part, and $z_{i}'=\tau(z_{i})$
if $d_{i}$ is decorated.
\end{enumerate}
\end{thm}

We remark that (\ref{eq:transitionmtx}) means $K$ is the transpose
of the matrix for $m\circ\Delta_{P(\tau)}$, relative to the basis
$\left\{ \frac{x}{\eta(x)}|x\in\calb_{n}\right\} $.
\begin{proof}
We follow the proof of \cite[Lem. 3.3]{descentoperatorchains}. According
to \cite[Th. 2.3]{descentoperatorchains}, to show $K$ is a transition
matrix, it suffices to show that $\eta:\calb_{n}\rightarrow\mathbb{R}$,
extended linearly to $\calh_{n}$, is an eigenvector of $\left(m\circ\Delta_{P(\tau)}\right)^{*}:\calhdual\rightarrow\calhdual$
of eigenvalue 1. Recall from Section \ref{subsec:hdos} that $\left(m\circ\Delta{}_{P(\tau)}\right)^{*}=m\circ\Delta{}_{P(\tau^{*})}$,
where $\tau^{*}$ is an involution on the dual basis $\calbdual$
of $\calhdual$. 

As in \cite[Lem. 3.3]{descentoperatorchains}, let $\bullet^{*}\in\calhdual_{1}$
denote the linear map on $\calh_{1}$ taking value $1$ on each element
of $\calb_{1}$. When restricted to $\calhn$, $\eta=(\bullet^{*})^{n}\in\calhdual_{n}$.
Since $\bullet^{*}$ is of degree 1, it is necessarily primitive,
so $\Delta_{D^{+}}((\bullet^{*})^{n})=\binom{n}{D^{+}}\left(\bullet^{*}\right)^{d_{1}}\otimes\dots\otimes\left(\bullet^{*}\right)^{d_{l}}$.
Now $\bullet^{*}$ is the sum of all elements in $\calbdual_{1}$,
and $\tau^{*}$ is an involution on $\calbdual_{1}$, so $\tau^{*}$
fixes $\bullet^{*}$. As a Hopf ambimorphism, $\tau^{*}$ also fixes
each $\left(\bullet^{*}\right)^{d_{i}}$, so $\Delta_{D}((\bullet^{*})^{n})=\Delta_{D^{+}}((\bullet^{*})^{n})=\binom{n}{D^{+}}\left(\bullet^{*}\right)^{d_{1}}\otimes\dots\otimes\left(\bullet^{*}\right)^{d_{l}}$,
and taking linear combination and then taking the product shows $m\circ\Delta_{P(\tau^{*})}((\bullet^{*})^{n})=((\bullet^{*})^{n})$,
as required.

The proof of the 4-step description of the chain is essentially the
same as that of its type A version \cite[Th. 3.4]{descentoperatorchains},
once we note that $\eta(\tau(z_{i}))=(\bullet^{*})^{n}(\tau(z_{i}))=\left[\tau^{*}(\bullet^{*})^{n}\right](z_{i})=(\bullet^{*})^{n}(z_{i})=\eta(z_{i})$,
since $\tau^{*}$ fixes $(\bullet^{*})^{n}$.
\end{proof}

\subsection{Stationary Distribution\label{subsec:Stationary-Distribution}}

One basic question when studying a Markov chain is to find its \emph{stationary
distributions}, which are functions $\pi:\calb\rightarrow\mathbb{R}$
satisfying 
\begin{equation}
\sum_{x\in\calb}\pi(x)K(x,y)=\pi(y).\label{eq:stationarydistributiondef}
\end{equation}
These are of interest as they include all possible limiting distributions:
if $\Prob(X_{t}=x)$ has a limit as $t\rightarrow\infty$, then this
limit must be $\pi(x)$ for some stationary $\pi$.

To describe the stationary distributions of chains defined in (\ref{eq:transitionmtx}),
let $\calb_{1}^{\tau}$ denote the subset of $\calb_{1}$ that is
fixed under $\tau$, and\textcolor{green}{{} $\calb_{1}^{-}$} the set
of averages over each size 2 $\tau$-orbit in $\calb_{1}$. (For example,
in the signed shuffle algebra with $\tau$ defined as in (\ref{eq:tau-shufflealg}),
$\calb_{1}^{-}$ would contain $\frac{1}{2}(1+\bar{1}),\frac{1}{2}(2+\bar{2}),$etc.,
and $\calb_{1}^{\tau}$ is empty.) Note that $\calb_{1}^{-}$ is also
fixed under $\tau$; indeed, a basis for $\calh_{1}^{\tau}$ is $\calb_{1}^{\tau}\cup\calb_{1}^{-}$.
Given any multiset $\{c_{1},\dots,c_{n}\}$ in $\calb_{1}^{\tau}\cup\calb_{1}^{-}$,
\textcolor{blue}{define $\pi_{c_{1},\dots,c_{n}}:\calbn\rightarrow\mathbb{R}$
by 
\[
\pi_{c_{1},\dots,c_{n}}(x):=\frac{\eta(x)}{n!^{2}}\sum_{\sigma\in\sn}\mbox{coefficient of }x\mbox{ in the product }c_{\sigma(1)}\dots c_{\sigma(n)}.
\]
Note that, if }$\calb_{1}$ contains a sole element, denoted $\bullet$,
then necessarily $\calb_{1}^{\tau}=\{\bullet\}$, $\calb_{1}^{-}=\emptyset$,
so \textcolor{blue}{the only possible such function }simplifies to\textcolor{blue}{
\[
\pi(x):=\frac{\eta(x)}{n!}\xi_{\bullet,\dots,\bullet}^{x}.
\]
}
\begin{thm}
\label{thm:stationary} Assume the conditions in Setup \ref{setup:chain},
and additionally that $P$ is non-zero on some $D$ with at least
two non-zero parts, and on some $D$ with at least one decorated part.
Then any stationary distribution for the Markov chain defined in Equation
(\ref{eq:transitionmtx}) is a unique linear combination of the $\pi_{c_{1},\dots,c_{n}}$,
over all multisets $\{c_{1},\dots,c_{n}\}$ in $\calb_{1}^{\tau}\cup\calb_{1}^{-}$.
Furthermore, $\sum_{x\in\calbn}\pi_{c_{1},\dots,c_{n}}(x)=1$. In
particular, if $\calb_{1}=\left\{ \bullet\right\} $, then $\pi$
is the unique stationary distribution.
\end{thm}

\begin{proof}
Recall that $K$ is the transpose of the matrix for $m\circ\Delta_{P(\tau)}$,
relative to the basis $\left\{ \frac{x}{\eta(x)}|x\in\calb_{n}\right\} $.
Thus the condition of (\ref{eq:stationarydistributiondef}) translates
to $\frac{1}{n!^{2}}\sum_{\sigma\in\sn}c_{\sigma(1)}\dots c_{\sigma(n)}$
being an eigenvector of eigenvalue 1 for $m\circ\Delta_{P(\tau)}$.
To check this, recall $c_{i}\in\calb_{1}^{\tau}\cup\calb_{1}^{-}\subseteq Prim(H)^{\tau}$,
so by (\ref{eq:descent-operator-on-product-of-primitives})
\[
m\circ\Delta_{D(\tau)}(c_{\sigma(1)}\dots c_{\sigma(n)})=\sum_{B_{1},\dots,B_{l(D)}}\left(\prod_{i\in B_{1}}c_{\sigma(i)}\right)\dots\left(\prod_{i\in B_{l(D)}}c_{\sigma(i)}\right),
\]
summing over all set-compositions $B$ compatible with $\lambda=(1,\dots,1),\barlambda=\emptyset,D$.
Each summand on the right is a product of $c_{\sigma(1)},\dots,c_{\sigma(n)}$
in some order, and there are $\binom{n}{D^{+}}$ terms, so when symmetrised,
$\sum_{\sigma\in\sn}c_{\sigma(1)}\dots c_{\sigma(n)}$ is an eigenvector
of eigenvalue $\binom{n}{D^{+}}$ for $m\circ\Delta_{D(\tau)}$. Taking
the linear combination of hyperoctahedral descent operators as in
(\ref{eq:mdeltaP}) gives the result.

The Theorem further claims that these $\frac{1}{n!^{2}}\sum_{\sigma\in\sn}c_{\sigma(1)}\dots c_{\sigma(n)}$
give a basis of the eigenspace of eigenvalue 1 for $m\circ\Delta_{P(\tau)}$,
as $\{c_{1},\dots,c_{n}\}$ ranges over all multisets in $\calb_{1}^{\tau}\cup\calb_{1}^{-}$.
To see the linear independence, note that, if the basis $\calp$ of
$(Prim\calh)^{\tau}$ contains $\calb_{1}^{\tau}\cup\calb_{1}^{-}$,
then the PBW basis constructed in and before Notation \ref{notn:pbwbasis}
contains $c_{1}\dots c_{n}$. Hence, after PBW straightening, the
highest length term in $\frac{1}{n!}\sum_{\sigma\in\sn}c_{\sigma(1)}\dots c_{\sigma(n)}$
is $c_{1}\dots c_{n}$, a different term for each choice of multiset
$\{c_{1},\dots,c_{n}\}$ in $\calb_{1}^{\tau}\cup\calb_{1}^{-}$,
showing linear independence. To see that these elements span the eigenspace
of eigenvalue 1, we separate into two claims as in the type A case
\cite[Th. 3.12]{descentoperatorchains}:
\begin{enumerate}
\item the only double-partition with $\sum_{D}\frac{P(D)}{\binom{n}{D^{+}}}\beta_{\lambda,\barlambda}^{D}=1$
is $\lambda=(1,\dots,1),\barlambda=\emptyset$ --- in the proof of
\cite[Th. 3.12]{descentoperatorchains}, it was shown that $\beta_{\lambda,\emptyset}^{D}\leq\binom{n}{D^{+}}$
with equality if and only if $D$ has a single non-zero part or $\lambda=(1,\dots,1)$.
(In short, this is because $\beta_{(1,\dots,1),\emptyset}^{D}$ by
definition counts \textcolor{blue}{the set-compositions of $\{1,2,\dots,n\}$
into $l(D)$ blocks such that block $i$ contains $d_{i}$ elements,
and} an alternate view of $\beta_{\lambda,\emptyset}^{D}$ is that
it counts \textcolor{blue}{such set-compositions} with the extra condition
\textcolor{blue}{that $1,2,\dots,\lambda_{1}$ are in the same block,
$\lambda_{1}+1,\lambda_{1}+2,\dots,\lambda_{1}+\lambda_{2}$ are in
the same block, and so on.) }Now $\beta_{\lambda,\barlambda}^{D}$
and $\beta_{\lambda\cup\barlambda,\emptyset}^{D}$ are respectively
a signed and unsigned enumeration of the same set-compositions, so
$\beta_{\lambda,\barlambda}^{D}\leq\beta_{\lambda\cup\barlambda,\emptyset}^{D}$
with equality if and only if $\barlambda=\emptyset$ or, in all set-compositions
compatible with $\lambda,\barlambda,D$, the signed integers $\bari$
can only be in blocks corresponding unsigned parts of $D$. In particular,
if $\lambda\cup\barlambda=(1,\dots,1)$, then this latter condition
can only happen if $D$ has no decorated parts. So $\beta_{\lambda,\barlambda}^{D}\leq\beta_{\lambda\cup\barlambda,\emptyset}^{D}\leq\binom{n}{D^{+}}$,
with equality if and only if $D$ has a single non-zero part, or $\lambda=\barlambda=(1,\dots,1)$
and $D$ has no decorated parts, or $\lambda=(1,\dots,1)$ and $\barlambda=\emptyset$.
The hypotheses in Theorem \ref{thm:stationary} exactly rules out
the first two equality cases for some $D$ in the sum $\sum_{D}\frac{P(D)}{\binom{n}{D^{+}}}\beta_{\lambda,\barlambda}^{D}$,
so $\sum_{D}\frac{P(D)}{\binom{n}{D^{+}}}\beta_{\lambda,\barlambda}^{D}\leq1$
with equality if and only if $\lambda=(1,\dots,1),\barlambda=\emptyset$.
\item the multiplicity of the eigenvalue $\sum_{D}\frac{P(D)}{\binom{n}{D^{+}}}\beta_{(1,\dots,1),\emptyset}^{D}$
is the number of multisets $\{c_{1},\dots,c_{n}\}$ in $\calb_{1}^{\tau}\cup\calb_{1}^{-}$
--- by Theorem \ref{thm:evalues}, the required multiplicity is the
coefficient of $x_{1}^{n}$ in $\prod_{i}(1-x_{i})^{-b_{i}}\prod_{i}(1-\bar{x}_{i})^{-\bar{b}_{i}}$,
which is $\binom{b_{1}+n-1}{n}$. From (\ref{eq:multiplicity1}) and
(\ref{eq:multiplicity2}) defining the numbers $b_{i}$ and $\bar{b_{i}}$,
it is clear that $b_{1}=\dim\calh_{1}^{\tau}=|\calb_{1}^{\tau}\cup\calb_{1}^{-}|$,
hence $\binom{b_{1}+n-1}{n}$ is the number of $n$-element multisets
in $\calb_{1}^{\tau}\cup\calb_{1}^{-}$.
\end{enumerate}
Finally, we check $\sum_{x\in\calbn}\pi_{c_{1},\dots,c_{n}}(x)=1$.
The key is the following identity, from the proof of the type A case
\cite[Th. 3.12]{descentoperatorchains}: 
\begin{equation}
\sum_{x\in\calbn}\eta(x)\xi_{d_{1},d_{2},\dots,d_{n}}^{x}=n!,\label{eq:sumetaxi}
\end{equation}
for any choice of $d_{1},\dots,d_{n}\in\calb_{1}$. Now, given $\{c_{1},\dots,c_{k}\}$
a multiset in $\calb_{1}^{\tau}$ and $\{c_{k+1},\dots,c_{n}\}$ a
multiset in $\calb_{1}^{-}$, let $c_{i}=\frac{1}{2}(d_{i}+d_{i}')$
for $i>k$, with $d_{i},d'_{i}\in\calb_{1}$. Then the coefficient
of $x$ in $c_{1}\dots c_{n}$ is  
\[
\frac{1}{2^{n-k}}\left(\xi_{c_{1},\dots,c_{k},d_{k+1},\dots,d_{n}}^{x}+\xi_{c_{1},\dots,c_{k},d'_{k+1},\dots,d_{n}}^{x}+\xi_{c_{1},\dots,c_{k},d_{k+1},d'_{k+2},\dots,d_{n}}^{x}+\xi_{c_{1},\dots,c_{k},d'_{k+1},d_{k+2}',\dots,d_{n}}^{x}+\dots\right)
\]
with $2^{n-k}$ terms in total, \textcolor{green}{running through
each combination of $d_{i}$ or $d_{i}'$}. Each of these $2^{n-k}$
terms, when multiplied by $\eta(x)$ and summed over $x\in\calbn$,
gives $n!$ by (\ref{eq:sumetaxi}), and the same is true when considering
the coefficient of $x$ in $c_{\sigma(1)}\dots c_{\sigma(n)}$ for
all $\sigma\in\sn$. 
\end{proof}

\section{Applications to Hyperoctahedral Riffle-Shuffling\label{sec:Card-Shuffling}}

This section applies the eigenvector formulas of Section \ref{subsec:oriffle-evector}
to analyse hyperoctahedral riffle-shuffles, as defined in Section
\ref{subsec:Hyperoctahedral-Riffle-Shuffles} below. Recall that each
formula depends on a basis of primitives, \textcolor{green}{which
Section \ref{subsec:lyndonword} will describe}, before full computation
of the eigenbasis in Section \ref{subsec:oriffle-ebasis-cards} and
its application to deduce the expected number of descents in Section
\ref{subsec:oriffle-descentpeak}.

\subsection{Hyperoctahedral Riffle-Shuffles\label{subsec:Hyperoctahedral-Riffle-Shuffles}}

For the rest of the paper, we study two families of hyperoctahedral
riffle-shuffles, the \emph{$a$-handed riffle-shuffle with rotation}
and \emph{$a$-handed riffle-shuffle with flip}. These are essentially
the Markov chains described in Theorem \ref{thm:chain-3step}, when
$\calh$ is the signed shuffle algebra of Example \ref{ex:shufflealg},
$\tau$ and $\tautilde$ are as defined in (\ref{eq:tau-shufflealg})
and (\ref{eq:tautilde-shufflealg}), and $m\circ\Delta_{P(\tau)}=\frac{1}{a^{n}}\oriftauplus$
or $\frac{1}{a^{n}}\oriftauminus$ (for rotation), or $\frac{1}{a^{n}}\oriftautildeplus$
or $\frac{1}{a^{n}}\oriftautildeminus$ (for flip). The factor of
$\frac{1}{a^{n}}$ is necessary in order that the resulting operator
can be written in the form in (\ref{eq:mdeltaP}), for a probability
distribution $P$ that sums to 1. For example, for $\frac{1}{a^{n}}\oriftauplus$,
the distribution is $P(D)=\binom{n}{D^{+}}$ for all $D$ with\textcolor{green}{{}
$a$ parts and odd parts positive, even parts negative}; and $P(D)=0$
otherwise. The case for the other three operators are similar.

We take the state space $\calbn$ to be all $2^{n}n!$ ways of ordering
$n$ distinct cards and tracking their orientations. Algebraically,
$\calbn$ is the set of words of length $n$, where, for each $i\in\{1,2,\dots,n\}$,
$i$ or $\bari$ appear exactly once. We may call these \emph{signed
permutations of $n$}, viewed in one-line notation. Technically, $\calbn$
is not a basis of $\calhn$, which should include decks with repeated
cards (words with repeated letters), but the span of $\calbn$ is
invariant under $m\circ\Delta_{P(\tau)}$ and $m\circ\Delta_{P(\tautilde)}$
for all $P$, so the theory of the previous sections can apply with
minor modifications. Note that $\eta(x)=1$ for all $x\in\calbn$.

The application of Theorem \ref{thm:chain-3step} to this situation
shows:
\begin{cor}
\label{cor:shuffle3step} Fix an integer $n$, and let $x,y$ be signed
permutations of $n$. For a fixed integer $a$, the matrix
\[
K(x,y):=\mbox{ coefficient of }y\mbox{ in }\frac{1}{a^{n}}\oriftauplus(x)\text{ (resp. }\frac{1}{a^{n}}\oriftauminus(x)\text{ )}
\]
is the transition matrix for the following card shuffle:

\begin{enumerate}[label=\arabic*.]
\item Choose a weak-composition $\left(d_{1},\dots,d_{a}\right)$ of $n$
according to the multinomial distribution, i.e. with probability $\frac{1}{a^{n}}\binom{n}{d_{1}\ \dots\ d_{a}}$.
\item Cut the deck into $a$ piles so the $i$th pile contains $d_{i}$
cards. 
\item Rotate the first, third, fifth, ... piles (resp. second, fourth, sixth,
... piles) by 180 degrees.
\item Uniformly choose one of the $\binom{n}{d_{1}\ \dots\ d_{a}}$ interleavings
of the $a$ piles. Equivalently, drop the cards one-by-one from the
bottom of one of the $a$ piles, chosen with probability proportional
to the current pile size.
\end{enumerate}
If we use $\tautilde$ in place of $\tau$, then Step 3 should flip
the piles upside down instead of rotating. \qed
\end{cor}

The equivalence in step 4 is proved in \cite[Lem. 1]{riffleshuffle}
regarding riffle-shuffles.

We would like to compute the expected value of some functions under
this shuffle, using the eigenvectors of the hyperoctahedral descent
operators as proved in Section \ref{subsec:oriffle-evector}. By Propositions
2.1 and 2.4 of \cite{descentoperatorchains} (reproduced below), such
calculations should take place on the\textcolor{green}{{} dual of the
algebra defining the chain}, namely the signed free associative algebra
in our case.
\begin{prop}
\cite[Prop. 2.1, 2.4]{descentoperatorchains} \label{prop:eigenfunction-expectation}
Suppose $K$ is the transition matrix for the Markov chain $\{X_{t}\}$
on the state space $\calb$. Let $\f:\calb\rightarrow\mathbb{R}$
be a (right) \emph{eigenfunction} of $\{X_{t}\}$ with eigenvalue
$\beta$, meaning $\sum_{y\in\calb}K(x,y)\f(y)=\beta\f(x)$. Then
the expected value of $\f$ is
\[
\Expect(\f(X_{t})|X_{0}=x_{0}):=\sum_{y\in\calb}K^{t}(x_{0},y)\f(y)=\beta^{t}\f(x_{0}).
\]
In the particular case where $\{X_{t}\}$ arises from the construction
of Theorem \ref{thm:chain-3step}, the right eigenfunctions $\f$
are in bijection with the eigenvectors $f\in\calh^{*}$ of the dual
map $m\circ\Delta_{P(\tau^{*})}$, through the vector space isomorphism
\[
\f(x):=\frac{1}{\eta(x)}f(x).
\]
 \qed
\end{prop}

\subsection{Lyndon Word Combinatorics for Signed Words\label{subsec:lyndonword}}

Let $\cala$ be an ordered alphabet, and $T(\cala)$ the free associative
algebra over $\cala$. The usual basis for $Prim(T(\cala))$ are the
``standard-bracketings of Lyndon words in $\cala$'' \cite[Ch. 5]{lothaire}.
The following explains this when $\cala=\{\bar{1}\preceq1\preceq\bar{2}\preceq2\preceq\dots\preceq\bar{N}\preceq N\}$,
with an important modification so that the output will be either $\tau$-invariant
or $\tau$-negating, and hence suitable for input into the eigenvector
formulas of Theorems \ref{thm:flipriffle-evector} and \ref{thm:rotriffleodd-evector}.

A word is \emph{Lyndon} if it is lexicographically smaller than all
its cyclic rearrangements. For example, $\bar{1}6\bar{7}\bar{2}$
is Lyndon, because $\bar{1}6\bar{7}\bar{2}\prec6\bar{7}\bar{2}\bar{1},\ \bar{7}\bar{2}\bar{1}6,\ \bar{2}\bar{1}6\bar{7}$.
In contrast, $\bar{5}3$ and $3\bar{5}3\bar{5}$ are not Lyndon.
\begin{rem*}
A word with distinct letters is Lyndon if and only if its starting
letter is minimal amongst all its letters.
\end{rem*}
\begin{defn}
\label{def:stdbrac} Given a Lyndon word $u$, its \emph{signed standard-bracketing}
$\stdbrac(u)$ is computed recursively as follows: if $u=i$ is a
single positive letter, set $\stdbrac(u):=i+\bari$. If $u=\bari$
is a single negative letter, set $\stdbrac(u):=i-\bari$. Otherwise,
write $u$ as the concatentation of Lyndon words $\leftu$ and $\rightu$
both non-empty such that $\rightu$ is of maximal length -- \cite[Prop. 5.1.3]{lothaire}
asserts that this ``standard factorisation'' is possible. Then define
\[
\stdbrac(u):=[\stdbrac(\leftu),\stdbrac(\rightu)].
\]
(These brackets denote the Lie bracket, $[z,w]:=zw-wz$.)
\end{defn}

\begin{rem*}
From the previous Remark, it is easy to see that, if $u$ has distinct
letters, then $\rightu$ begins precisely with the second-minimal
letter of $u$.
\end{rem*}
\begin{example}
\label{ex: stdbrac}
\begin{align*}
\stdbrac(\bar{1}6\bar{7}\bar{2}): & =[\stdbrac(\bar{1}6\bar{7}),\stdbrac(\bar{2})]\\
 & =[\stdbrac(\bar{1}6\bar{7}),2-\bar{2}]\\
 & =[[1-\bar{1},\stdbrac(6\bar{7})],2-\bar{2}]\\
 & =[[1-\bar{1},[6+\bar{6},7-\bar{7}]],2-\bar{2}]\\
 & =1672-16\bar{7}2+1\bar{6}72-1\bar{6}\bar{7}2-1762-17\bar{6}2+1\bar{7}62+1\bar{7}\bar{6}2\\
 & \quad-\bar{1}672+\bar{1}6\bar{7}2-\bar{1}\bar{6}72+\bar{1}\bar{6}\bar{7}2+\bar{1}762+\bar{1}7\bar{6}2-\bar{1}\bar{7}62-\bar{1}\bar{7}\bar{6}2\\
 & \quad-6712+6\bar{7}12-\bar{6}712+\bar{6}\bar{7}12+7612+7\bar{6}12-\bar{7}612-\bar{7}\bar{6}12\\
 & \quad+67\bar{1}2-6\bar{7}\bar{1}2+\bar{6}7\bar{1}2-\bar{6}\bar{7}\bar{1}2-76\bar{1}2-7\bar{6}\bar{1}2+\bar{7}6\bar{1}2+\bar{7}\bar{6}\bar{1}2\\
 & \quad-167\bar{2}+16\bar{7}\bar{2}-1\bar{6}7\bar{2}+1\bar{6}\bar{7}\bar{2}+\dots\mbox{ (28 more terms, analogous to those above)}\\
 & \quad-2167+216\bar{7}-21\bar{6}7+21\bar{6}\bar{7}+\dots\mbox{ (28 more terms)}\\
 & \quad+\bar{2}167-\bar{2}16\bar{7}+\bar{2}1\bar{6}7-\bar{2}1\bar{6}\bar{7}+\dots\mbox{ (28 more terms)}.
\end{align*}
\end{example}

Proposition \ref{prop:bracket-tau} and Corollary \ref{cor:stdbrac-invariantornegating}
explain how to determine if $\stdbrac(u)$ is $\tau$-invariant or
$\tau$-negating.
\begin{prop}
\label{prop:bracket-tau} Let $\calh$ be any Hopf algebra. 

For $\tau:\calh\rightarrow\calh$ a linear involution that is an algebra
morphism,
\[
[\calh^{\tau},\calh^{\tau}]\subseteq\calh^{\tau},\qquad[\calh^{\tau},\calh^{-\tau}]\subseteq\calh^{-\tau},\qquad[\calh^{-\tau},\calh^{-\tau}]\subseteq\calh^{\tau}.
\]

For $\tautilde:\calh\rightarrow\calh$ a linear involution that is
an algebra antimorphism,
\[
[\calh^{\tautilde},\calh^{\tautilde}]\subseteq\calh^{-\tautilde},\qquad[\calh^{\tautilde},\calh^{-\tautilde}]\subseteq\calh^{\tautilde},\qquad[\calh^{-\tautilde},\calh^{-\tautilde}]\subseteq\calh^{-\tautilde}.
\]
\end{prop}

\begin{proof}
For $x,y\in\calh^{\tau}$ and $\bar{x},\bar{y}\in\calh^{-\tau}$:
\begin{align*}
\tau([x,y]) & =\tau(xy-yx)=\tau(x)\tau(y)-\tau(y)\tau(x)=xy-yx=[x,y];\\
\tau([x,\bar{y}]) & =\tau(x\bar{y}-\bar{y}x)=\tau(x)\tau(\bar{y})-\tau(\bar{y})\tau(x)=x(-\bar{y})-(-\bar{y})x=-[x,\bar{y}];\\
\tau([\bar{x},\bar{y}]) & =\tau(\bar{x}\bar{y}-\bar{y}\bar{x})=\tau(\bar{x})\tau(\bar{y})-\tau(\bar{y})\tau(\bar{x})=(-\bar{x})(-\bar{y})-(-\bar{y})(-\bar{x})=[\bar{x},\bar{y}].
\end{align*}

For $x,y\in\calh^{\tautilde}$ and $\bar{x},\bar{y}\in\calh^{-\tautilde}$:
\begin{align*}
\tautilde([x,y]) & =\tautilde(xy-yx)=\tautilde(y)\tautilde(x)-\tautilde(x)\tautilde(y)=yx-xy=-[x,y];\\
\tautilde([x,\bar{y}]) & =\tautilde(x\bar{y}-\bar{y}x)=\tautilde(\bar{y})\tautilde(x)-\tautilde(x)\tautilde(\bar{y})=(-\bar{y})x-x(-\bar{y})=[x,\bar{y}];\\
\tautilde([\bar{x},\bar{y}]) & =\tautilde(\bar{x}\bar{y}-\bar{y}\bar{x})=\tautilde(\bar{y})\tautilde(\bar{x})-\tautilde(\bar{x})\tautilde(\bar{y})=(-\bar{y})(-\bar{x})+(-\bar{x})(-\bar{y})=[\bar{x},\bar{y}].
\end{align*}
\end{proof}

\begin{cor}
\label{cor:stdbrac-invariantornegating} Let $u$ be a Lyndon word
in the alphabet $\cala=\{\bar{1}\preceq1\preceq\bar{2}\preceq2\preceq\dots\preceq\bar{N}\preceq N\}$,
and $T(\cala)$ be the free associative algebra over $\cala$, i.e.
the signed free associative algebra. Then 
\begin{align*}
\stdbrac(u)\in(PrimT(\cala))^{\tau} & \iff u\mbox{ has an even number of negative letters};\\
\stdbrac(u)\in(PrimT(\cala))^{-\tau} & \iff u\mbox{ has an odd number of negative letters};\\
\stdbrac(u)\in(PrimT(\cala))^{\tautilde} & \iff u\mbox{ has an odd number of positive letters};\\
\stdbrac(u)\in(PrimT(\cala))^{-\tautilde} & \iff u\mbox{ has an even number of positive letters}.
\end{align*}
\end{cor}

\begin{proof}
We prove the last line. The other three are entirely analogous.

Apply induction on the length of $u$. In the base case where $u$
is a single letter, then it has an even number of positive letters
precisely when $u$ is a negative letter, e.g. $u=\bari$. Then
$\stdbrac(u)=i-\bari\in(PrimT(\cala))^{-\tautilde}$ as required.

Now suppose $u$ has at least two letters. If $u$ has an even number
of positive letters, then there are two possibilities for $\leftu$
and $\rightu$:

\begin{itemize}
\item $\leftu$ and $\rightu$ both have an odd number of positive letters,
so by inductive hypothesis $\stdbrac(\leftu),\stdbrac(\rightu)\in(PrimT(\cala))^{\tautilde}$.
So 
\[
\stdbrac(u)=[\stdbrac(\leftu),\stdbrac(\rightu)]\subseteq[(PrimT(\cala))^{\tautilde},(PrimT(\cala))^{\tautilde}]\subseteq(PrimT(\cala))^{-\tautilde},
\]
using Proposition \ref{prop:bracket-tau} for the last inclusion.
\item $\leftu$ and $\rightu$ both have an even number of positive letters,
so $\stdbrac(\leftu),\stdbrac(\rightu)\in(PrimT(\cala))^{-\tautilde}$.
So 
\[
\stdbrac(u)\in[(PrimT(\cala))^{-\tautilde},(PrimT(\cala))^{-\tautilde}]\subseteq(PrimT(\cala))^{-\tautilde}
\]
also.
\end{itemize}
\end{proof}

\subsection{Eigenbasis Algorithm\label{subsec:oriffle-ebasis-cards}}

The previous section gave a signed standard-bracketing algorithm which
converts each Lyndon word to a $\tau$-invariant or $\tau$-negating
primitive. Theorems \ref{thm:flipriffle-evector} and \ref{thm:rotriffleodd-evector}
associate an eigenvector for each multiset of such primitives, but
it's more convenient to index the eigenvectors by words. The latter
indexing may be achieved by a bijection sending a word $w$ to its
\emph{decreasing Lyndon factorisation} $(u_{1},\dots,u_{k'})$, where
$w$ is the concatenation of the Lyndon words $u_{i}$ satisfying
$u_{1}\succeq u_{2}\succeq\dots\succeq u_{k'}$ in lexicographic order.
By the Remark before Definition \ref{def:stdbrac}, if $w$ has distinct
letters, then the $u_{i}$ begin at precisely the left-to-right minima
of $w$. For example, the decreasing Lyndon factorisation of $\bar{4}35\bar{1}6\bar{7}\bar{2}$
is $(\bar{4},35,\bar{1}6\bar{7}\bar{2})$. 

Hence the full eigenvector algorithm is as follows:

\begin{algo}\label{algo:shuffleevector} To associate to a word $w$
an eigenvector of $\oriftauplus$ or $\oriftauminus$ for odd $a$,
or of $\oriftautildeplus$ or $\oriftautildeminus$ for any $a$,
on the signed free associative algebra:
\begin{enumerate}[label=\arabic*.]
\item Take the decreasing Lyndon factorisation $u_{1},\dots,u_{k'}$ of
$w$.
\item Calculate $\stdbrac(u_{i})$ for each Lyndon factor using Definition
\ref{def:stdbrac}.
\item Use Corollary \ref{cor:stdbrac-invariantornegating} to determine
whether each result from Step 2 is $\tau$-invariant or $\tau$-negating
-- relabel the $\tau$-invariant ones as $p_{1},\dots,p_{k}$ and
the $\tau$-negating ones as $\barp_{1}.\dots,\barp_{\bark}$.
\item Assemble the $p_{i}$ and $\barp_{i}$ according to Theorem \ref{thm:flipriffle-evector}
or \ref{thm:rotriffleodd-evector} for the \textcolor{green}{desired
operator}. 
\end{enumerate}
By Theorems \ref{thm:flipriffle-evector}.iii and \ref{thm:rotriffleodd-evector}.iii,
the eigenvectors associated with all words in the above manner form
a basis of the free associative algebra $\calh$.\end{algo}
\begin{example}
\label{ex:rotriffle-evectoralgo} Let $w=\bar{4}35\bar{1}6\bar{7}\bar{2}$.
We calculate the corresponding eigenvector for $\orif\tau_{3}^{+}$.

\begin{enumerate}[label=\arabic*.]
\item The Lyndon factors of $w$ are $u_{1}=\bar{4}$, $u_{2}=35$, $u_{3}=\bar{1}6\bar{7}\bar{2}$.
\item Their standard bracketings are, respectively, $4-\bar{4}$, $[3+\bar{3},5+\bar{5}]$
and $[[1-\bar{1},[6+\bar{6},7-\bar{7}]],2-\bar{2}]$.
\item $\bar{4}$ has one negative letter, so $4-\bar{4}$ is $\tau$-negating.
Call this $\barp_{1}$.

$35$ has no negative letters, so $[3+\bar{3},5+\bar{5}]$ is $\tau$-invariant.
Call this $p_{1}$.

$\bar{1}6\bar{7}\bar{2}$ has three negative letters, so $[[1-\bar{1},[6+\bar{6},7-\bar{7}]],2-\bar{2}]$
is $\tau$-negating. Call this $\barp_{2}$.
\item Since $a=3$ is odd, Theorem \ref{thm:rotriffleodd-evector} applies.
By inputting $p_{1},\barp_{1},\barp_{2}$ above into (\ref{eq:rotriffle-evector}):
\begin{align*}
 & (4-\bar{4})[[1-\bar{1},[6+\bar{6},7-\bar{7}]],2-\bar{2}][3+\bar{3},5+\bar{5}]+(4-\bar{4})[3+\bar{3},5+\bar{5}][[1-\bar{1},[6+\bar{6},7-\bar{7}]],2-\bar{2}]\\
+ & [[1-\bar{1},[6+\bar{6},7-\bar{7}]],2-\bar{2}][3+\bar{3},5+\bar{5}](4-\bar{4})+[3+\bar{3},5+\bar{5}][[1-\bar{1},[6+\bar{6},7-\bar{7}]],2-\bar{2}](4-\bar{4})
\end{align*}
is an eigenvector for $\orif\tau_{3}^{+}$, of eigenvalue $3$.
\end{enumerate}
\end{example}

\begin{example}
\label{ex:flipriffle-evectoralgo} We calculate the eigenvector for
$\orif\tautilde_{2}^{+}$ and $\orif\tautilde_{3}^{-}$ corresponding
to $w=\bar{4}\bar{3}5\bar{1}6\bar{7}\bar{2}$. Steps 1 and 2 are similar
to Example \ref{ex:rotriffle-evectoralgo} above.

\begin{enumerate}
\item[3.] $\bar{4}$ has no positive letters, so $4-\bar{4}$ is $\tautilde$-negating.
Call this $\barp_{1}$.

$\bar{3}5$ has one positive letter, so $[3-\bar{3},5+\bar{5}]$ is
$\tautilde$-invariant. Call this $p_{1}$.

$\bar{1}6\bar{7}\bar{2}$ has one positive letter, so $[[1-\bar{1},[6+\bar{6},7-\bar{7}]],2-\bar{2}]$
is $\tautilde$-invariant. Call this $p_{2}$.
\item[4.]  Since $a=2$ is even, input $p_{1},p_{2,}\barp_{1}$ above into
the second formula in Theorem \ref{thm:flipriffle-evector}.ia:
\begin{align*}
 & (4-\bar{4})\left([3-\bar{3},5+\bar{5}][[1-\bar{1},[6+\bar{6},7-\bar{7}]],2-\bar{2}]\right.\\
 & \hphantom{(4-\bar{4})}\left.\qquad+[3-\bar{3},5+\bar{5}][[1-\bar{1},[6+\bar{6},7-\bar{7}]],2-\bar{2}]\right)
\end{align*}
 is an eigenvector for $\orif\tautilde_{2}^{+}$, of eigenvalue 0.
And, for $\orif\tautilde_{3}^{-}$, input $p_{1},p_{2,}\barp_{1}$
into Theorem \ref{thm:flipriffle-evector}.ic:
\begin{align*}
 & (4-\bar{4})\left([3-\bar{3},5+\bar{5}][[1-\bar{1},[6+\bar{6},7-\bar{7}]],2-\bar{2}]\right.\\
 & \hphantom{(4-\bar{4})}\left.\qquad+[3-\bar{3},5+\bar{5}][[1-\bar{1},[6+\bar{6},7-\bar{7}]],2-\bar{2}]\right)\\
 & +\left([3-\bar{3},5+\bar{5}][[1-\bar{1},[6+\bar{6},7-\bar{7}]],2-\bar{2}]\right.\\
 & \hphantom{+}\left.\qquad+[3-\bar{3},5+\bar{5}][[1-\bar{1},[6+\bar{6},7-\bar{7}]],2-\bar{2}]\right)(4-\bar{4})
\end{align*}
is an eigenvector of eigenvalue $-3$.
\end{enumerate}
\end{example}

\subsection{Multiplicities of Eigenvalues\label{subsec:oriffle-evaluesequence}}

Because $\calbn$ is not a basis of $\calhn$ (as we're ignoring decks
with repeated cards), Proposition \ref{prop:oriffle-evalues} does
not apply to determine the multiplicities of the hyperoctahedral riffle-shuffles.
Instead, we enumerate the signed permutations \textcolor{green}{with
the required type of Lyndon factors as listed in Corollary \ref{cor:stdbrac-invariantornegating}. }
\begin{thm}
\textcolor{green}{\label{thm:oriffle-evaluemultiplicity} }Write $[g]f$
to mean the cofficient of the monomial $g$ in the power series $f$.
The (algebraic) multiplicities of the eigenvalues of the hyperoctahedral
riffle-shuffles are as given in Table \ref{tab:oriffle-evaluemultiplicities}.
\end{thm}

\begin{table}
\centering{}%
\begin{tabular}{|c|c|c|c|}
\hline 
parity of $a$ & shuffle & eigenvalue & multiplicity\tabularnewline
\hline 
\hline 
\multirow{2}{*}{even} & $\frac{1}{a^{n}}\oriftauplus$ or $\frac{1}{a^{n}}\oriftauminus$ & $a^{k-n}$ & $[x^{k}]x(x+2)(x+4)\dots(x+2n-2)$\tabularnewline
\cline{3-4} \cline{4-4} 
 & or $\frac{1}{a^{n}}\oriftautildeplus$ or $\frac{1}{a^{n}}\oriftautildeminus$ & 0 & $2^{n}n!-\frac{(2n)!}{2^{n}n!}$\tabularnewline
\hline 
\multirow{3}{*}{odd} & \multirow{1}{*}{$\frac{1}{a^{n}}\oriftauplus$ or $\frac{1}{a^{n}}\oriftautildeplus$ } & $a^{k-n}$ & $[x^{k}](x+1)(x+3)\dots(x+2n-1)$\tabularnewline
\cline{2-4} \cline{3-4} \cline{4-4} 
 & \multirow{2}{*}{$\frac{1}{a^{n}}\oriftauminus$ or $\frac{1}{a^{n}}\oriftautildeminus$} & $a^{k-n}$ & $[x^{k}](x+n-1)(x+1)(x+3)\dots(x+2n-3)$\tablefootnote{This is a signless version of OEIS sequence A039762}\tabularnewline
\cline{3-4} \cline{4-4} 
 &  & $-a^{k-n}$ & $[x^{k}]n(x+1)(x+3)\dots(x+2n-3)$\tabularnewline
\hline 
\end{tabular}\caption{\label{tab:oriffle-evaluemultiplicities}Multiplicities of the eigenvalues
of the hyperoctahedral riffle-shuffles }
\end{table}

To prove these multiplicities, let $c(n,k)$ be the \emph{signless
Stirling number of the first kind}, i.e. the number of permutations
of $n$ with $k$ left-to-right minima. Recall that, in a word with
distinct letters, the Lyndon factors start precisely at the left-to-right
minima, hence $c(n,k)$ is also the number of permutations of $n$
with $k$ Lyndon factors. Important for deducing the generating functions
of the multiplicities is the following generating function of $c(n,k)$
\begin{equation}
\sum_{k}c(n,k)x^{k}=x(x+1)\dots(x+n-1).\label{eq:stirlinggenfn}
\end{equation}

We define a hyperoctahedral analogue of $c(n,k)$: let $\mathcal{C}^{+}(n,k,\bark)$
(resp. $\mathcal{C^{-}}(n,k,\bark)$) be the set of signed permutations
of $n$ having $k$ Lyndon factors with an odd number of positive
(resp. negative) letters, and $\bark$ Lyndon factors with an even
number of positive (resp. negative) letters. Note that changing the
sign of all letters gives a bijection between $\mathcal{C}^{+}(n,k,\bark)$
and $\mathcal{C^{-}}(n,k,\bark)$, so let $C(n,k,\bark)$ denote their
common size. Observe $C(n,k,\bark)=C(n,\bark,k)$, by changing the
sign of the first letter in each factor.
\begin{lem}
\label{lem:bstirling-astirling}The hyperoctahedral Stirling numbers
are related to the type A Stirling numbers by
\[
C(n,k,\bark)=2^{n-k-\bark}c(n,k+\bark)\binom{k+\bark}{k}.
\]
\end{lem}

\begin{proof}
Given $k$ and $\bark$, fix any sequence of $k$ positive signs and
$\bark$ negative signs. There are $\binom{k+\bark}{k}$ such sequences.
We show below that, for any such sequence, there are $2^{n-k-\bark}c(n,k+\bark)$
signed permutations whose Lyndon factors $u_{i}$ have an odd number
(resp. even number) of positive letters if the $i$th term of the
sequence is positive (resp. negative).

First, there are $c(n,k+\bark)$ permutations with $k+\bark$ Lyndon
factors in total. The factor of $2^{n-k-\bark}$ then enumerates the
signs that can be assigned: within each Lyndon factor, all but the
last letter can have arbitrary sign; the sign of the last letter is
then determined by the requirement that this factor have an odd (or
even) number of positive letters.
\end{proof}
\begin{lem}
\label{lem:bstirling-recursion} The hyperoctahedral Stirling numbers
satisfy the recursion
\begin{equation}
C(n,k,\bark)=C(n-1,k-1,\bark)+C(n-1,\bark-1)+2(n-1)C(n-1,k,\bark).\label{eq:bstirlingrelation}
\end{equation}
\end{lem}

\begin{proof}
We construct a bijection \textcolor{green}{between the relevant sets}.
Each element of $\mathcal{C}^{+}(n,k,\bark)$ is obtained in precisely
one of the following manners:
\begin{itemize}
\item Adding $n$ to the start of an element of $\mathcal{\mathcal{C}^{+}}(n-1,k-1,\bark)$
- $n$ then becomes a new left-to-right minima, i.e. a new Lyndon
factor with one positive letter.
\item Adding $\bar{n}$ to the start of an element of $\mathcal{\mathcal{C}^{+}}(n-1,k,\bark-1)$
- $\bar{n}$ is then a new Lyndon factor with one negative and no
positive letters.
\item To an element of $\mathcal{\mathcal{C}^{+}}(n-1,k,\bark)$, add $\bar{n}$
between any two letters, or at the end. This does not produce any
new left-to-right minima, so $\bar{n}$ is inserted into one of the
existing Lyndon factors, and does not change its parity of its positive
letters.
\item To an element of $\mathcal{\mathcal{C}^{+}}(n-1,k,\bark)$, add $n$
between any two letters, or at the end, and also change the sign of
the preceding letter, which is necessarily within the same Lyndon
factor. This sign change ensures that the parity of positive letters
remains unchanged in this factor.
\end{itemize}
\end{proof}
\begin{proof}[Proof of Theorem \ref{thm:oriffle-evaluemultiplicity}]
 Fix $a$ even. From Theorems \ref{thm:rotriffleeven-evector} and
\ref{thm:flipriffle-evector}, and Corollary \ref{cor:stdbrac-invariantornegating},
the multiplicity of the eigenvalue $a^{k-n}$ is $|\mathcal{C}^{-}(n,0,k)|$
for $\frac{1}{a^{n}}\oriftauplus$ and $\frac{1}{a^{n}}\oriftauminus$,
and $|\mathcal{C}^{+}(n,k,0)|$ for $\frac{1}{a^{n}}\oriftautildeplus$
and $\frac{1}{a^{n}}\oriftautildeminus$. By the symmetry remarks
before Lemma \ref{lem:bstirling-astirling}, their common generating
function is
\begin{align*}
\sum_{k}C(n,k,0)x^{k} & =\sum_{k}2^{n-k}c(n,k)x^{k}\\
 & =2^{n}\sum_{k}c(n,k)\left(\frac{x}{2}\right)^{k}\\
 & =2^{n}\frac{x}{2}\left(\frac{x}{2}+1\right)\dots\left(\frac{x}{2}+n-1\right)\\
 & =x(x+2)(x+4)\dots(x+2n-2).
\end{align*}
0 is the only other eigenvalue. So that the multiplicities of all
eigenvalues sum to $2^{n}n!$, the multiplicity of 0 must be $2^{n}n!-1\cdot3\cdot\dots\cdot(2n-1)$.

Now consider $\frac{1}{a^{n}}\oriftautildeminus$ for $a$ odd. From
Theorems \ref{thm:rotriffleodd-evector} and Corollary \ref{cor:stdbrac-invariantornegating},
the multiplicity of the eigenvalue $a^{k-n}$ (resp. $-a^{k-n}$)
is $\sum_{\bark\text{ even}}|\mathcal{C}^{+}(n,k,\bark)|$ (resp.
$\sum_{\bark\text{ odd}}|\mathcal{C}^{+}(n,k,\bark)|$). It can be
proved by induction on $n$ that 
\begin{align*}
\sum_{\bark\text{ odd}}C(n,k,\bark) & =n\sum_{\bark}C(n-1,k,\bark),\\
\sum_{\bark\text{ even}}C(n,k,\bark) & =(n-1)\sum_{\bark}C(n-1,k,\bark)+\sum_{\bark}C(n-1,k-1,\bark);
\end{align*}
briefly, use (\ref{eq:bstirlingrelation}) to rewrite the left hand
side in terms of $C(n-1,-,-)$, apply the inductive hypothesis to
obtain sums of $C(n-2,-,-)$, over all values of $\bark$, then collect
terms using (\ref{eq:bstirlingrelation}) to obtain terms of the form
$C(n-1,-,-$). Then the claimed generating functions come from observing
that
\begin{align*}
\sum_{k,\bark}C(n-1,k,\bark)x^{k} & =\sum_{k,\bark}2^{n-1-k-\bark}c(n-1,k+\bark)\binom{k+\bark}{k}x^{k}\\
 & =\sum_{j}2^{n-1-j}c(n-1,j)(x+1)^{j}\\
 & =2^{n-1}\sum_{k}c(n-1,j)\left(\frac{x+1}{2}\right)^{j}\\
 & =2^{n-1}\frac{x+1}{2}\left(\frac{x+1}{2}+1\right)\dots\left(\frac{x+1}{2}+n-2\right)\\
 & =(x+1)(x+3)\dots(x+2n-3).
\end{align*}

The multiplicity of $a^{k-n}$ for $\frac{1}{a^{n}}\oriftautildeplus$
is the sum of the multiplicities of $a^{k-n}$ and $-a^{k-n}$ for
$\frac{1}{a^{n}}\oriftautildeminus$. The cases of $\frac{1}{a^{n}}\oriftauminus$
and $\frac{1}{a^{n}}\oriftauplus$ follow from the symmetry remarks
before Lemma \ref{lem:bstirling-astirling}.
\end{proof}

\subsection{Subdominant Eigenvectors and Expectations of Descents \label{subsec:oriffle-descentpeak}}

\textcolor{green}{The paper \cite[Exs. 5.8-5.9, Prop. 5.10]{hopfpowerchains}
identified} the eigenvectors of type A riffle-shuffles with large
eigenvalues, and expressed some of them in terms of two classical
permutation statistics, namely descents and peaks. This section dervies
one hyperoctahedral analogue.

First, to clarify what is meant by ``large'' eigenvalue:
\begin{defn}
An eigenvalue $\lambda$ of a Markov chain is \emph{subdominant} if
$|\lambda|<1$ and $|\lambda|>|\lambda'|$ for all other eigenvalues
$\lambda'\neq1,-1$.
\end{defn}

Informally, subdominant eigenvalues have maximal absolute value after
1. Their values and corresponding eigenspaces have the largest influence
on the convergence rate.

\textcolor{green}{The corresponding eigenvectors are expressed in
terms of the following terminology:}
\begin{defn}
\label{def:consecutivesubword} If $w=w_{1}\dots w_{n}$ where $w_{i}$
denotes the $i$th letter of $w$, then a \emph{consecutive subword}
of $w$ is a word of the form $w_{i}w_{i+1}\dots w_{j}$ for some
$i<j$. 
\end{defn}

For example, if $w=435\bar{1}6\bar{7}\bar{2}$, then $35\bar{1}$
and $6\bar{7}\bar{2}$ are consecutive subwords of $w$, but $36\mbox{\ensuremath{\bar{7}}}$
is not. 
\begin{defn}
\label{def:descent} Given a word $w$, its consecutive subword $ij$
is a \emph{descent} if $i>j$. This applies whether $i$ and $j$
are positive and negative letters. Let $\des(w)$ denote the number
of descents in $w$.
\end{defn}

For example, the descents of $435\bar{1}6\bar{7}\bar{2}$ are $43$,
$5\bar{1}$, and $6\bar{7}$, so $\des(w)=3$. Notice that this differs
from the definition of descent in Coxeter type B, where additionally
the first letter is a descent if it is negative. 

\begin{notn}\label{notn:subdominant-efns}For $i<j$, define 
\begin{equation}
\f_{ij}^{+}(w):=\begin{cases}
1 & \mbox{ if any of }ij,\ \bari\barj\mbox{ is a consecutive subword in }w;\\
-1 & \mbox{ if any of }ji,\ \barj\bari\mbox{ is a consecutive subword in }w;\\
0 & \mbox{ otherwise},
\end{cases}\label{eq:subdominantevector-fplus}
\end{equation}
and
\begin{equation}
\f_{ij}^{-}(w):=\begin{cases}
1 & \mbox{ if any of }\bari j,\ i\barj\mbox{ is a consecutive subword in }w;\\
-1 & \mbox{ if any of }j\bari,\ \barj i\mbox{ is a consecutive subword in }w;\\
0 & \mbox{ otherwise},
\end{cases}\label{eq:subdominantevector-fminus}
\end{equation}
And for $i\neq j$, define 
\begin{equation}
\ftilde_{ij}(w):=\begin{cases}
1 & \mbox{ if any of }ij,\ \bari j,\ \barj i,\ \barj\bari\mbox{ is a consecutive subword in }w;\\
-1 & \mbox{ if any of }ji,\ j\bari,\ i\barj,\ \bari\barj\mbox{ is a consecutive subword in }w;\\
0 & \mbox{ otherwise}.
\end{cases}\label{eq:subdominantevector-ftilde}
\end{equation}
And for all $i$, define 
\begin{equation}
\g_{i}(w):=\begin{cases}
1 & \mbox{ if }i\mbox{ is the first or last letter of }w;\\
-1 & \mbox{ if }\bari\mbox{ is the first or last letter of }w;\\
0 & \mbox{ otherwise}.
\end{cases}\label{eq:subdominantevector-g}
\end{equation}
\end{notn}
\begin{prop}
\label{prop:subdominant-efns} The subdominant eigenvalues of the
hyperoctahedral riffle-shuffles, and a basis of associated eigenfunctions,
are as given in Table \ref{tab:subdominanteigen}.
\end{prop}

\begin{table}
\centering{}%
\begin{tabular}{|c|c|c|c|}
\hline 
parity of $a$ & shuffle & subdominant eigenvalues & basis of eigenfunctions\tabularnewline
\hline 
\hline 
even & $\frac{1}{a^{n}}\oriftauplus$ or $\frac{1}{a^{n}}\oriftauminus$  & $1/a$ & $\{\f_{ij}^{+}\}\cup\{\f_{ij}^{-}\}$\tabularnewline
\hline 
\multirow{3}{*}{odd} & $\frac{1}{a^{n}}\oriftauplus$  & $1/a$ & $\{\f_{ij}^{+}\}\cup\{\f_{ij}^{-}\}\cup\{\g_{i}\}$ \tabularnewline
\cline{2-4} \cline{3-4} \cline{4-4} 
 & \multirow{2}{*}{$\frac{1}{a^{n}}\oriftauminus$} & $1/a$ & $\{\f_{ij}^{+}\}\cup\{\f_{ij}^{-}\}$ \tabularnewline
\cline{3-4} \cline{4-4} 
 &  & $-1/a$ & $\{\g_{i}\}$ \tabularnewline
\hline 
even & $\frac{1}{a^{n}}\oriftautildeplus$ or $\frac{1}{a^{n}}\oriftautildeminus$ & $1/a$ & $\{\ftilde_{ij}\}$ \tabularnewline
\hline 
\multirow{3}{*}{odd} & $\frac{1}{a^{n}}\oriftautildeplus$  & $1/a$ & $\{\ftilde_{ij}\}\cup\{\g_{i}\}$\tabularnewline
\cline{2-4} \cline{3-4} \cline{4-4} 
 & \multirow{2}{*}{$\frac{1}{a^{n}}\oriftautildeminus$ } & $1/a$ & $\{\ftilde_{ij}\}$ \tabularnewline
\cline{3-4} \cline{4-4} 
 &  & $-1/a$ & $\{\g_{i}\}$\tabularnewline
\hline 
\end{tabular}\caption{\label{tab:subdominanteigen}Subdominant eigenvalues for hyperoctahedral
riffle-shuffles, and bases for their eigenfunctions}
\end{table}

\begin{prop}
\label{prop:descentefn} Fix a positive integer $a$, and let $\{X_{t}\}$
denote $a$-handed riffle-shuffle with flip, i.e. the Markov chain
associated to $\frac{1}{a^{n}}\oriftautildeplus$ or $\frac{1}{a^{n}}\oriftautildeminus$
on $\calbn$ according to Section \ref{subsec:Hyperoctahedral-Riffle-Shuffles}.
Then the ``normalised number of descents'' functions $w\mapsto\des(w)-\frac{n-1}{2}$
is an eigenfunction of $\{X_{t}\}$ of eigenvalue $1/a$. Hence 
\[
\Expect(\des X_{t}|X_{0}=w_{0})=(1-a^{-t})\frac{n-1}{2}+a^{-t}\des(w_{0}).
\]
\end{prop}

\begin{proof}
We show that 
\[
\des(w)-\frac{n-1}{2}=\sum_{i<j}\ftilde_{ij}(w)
\]
for all $w\in\calbn$. Note that the subwords which cause $\ftilde_{ij}(w)=-1$,
namely $ji$, $j\bari$, $i\barj$ and $\bari\barj$, are all descents
if $i<j$. And every descent of $w$ is of one of these 4 types, for
some $i<j$. Opposingly, all consecutive subwords of two letters that
are not descents (i.e. ``ascents'') contribute 1 to $\ftilde_{ij}(w)$
for some $i<j$. Hence $\sum_{i<j}\ftilde_{ij}(w)$ is the difference
between the number of ascents and the number of descents in $w$.
Since each pair of consecutive letters in $w$ is either an ascent
or a descent, the number of ascents and descents together total $n-1$.
Hence this difference is $(n-1-\des(w))-\des(w)$. 

As for the expectation assertion: by the linearity of expectations,
then by Proposition \ref{prop:eigenfunction-expectation},
\begin{align*}
\Expect(\des X_{t}|X_{0}=w_{0}) & =\Expect(\sum_{i<j}\ftilde_{ij}(X_{t})|X_{0}=w)+\Expect\left(\frac{n-1}{2}\middle|X_{0}=w\right)\\
 & =a^{-t}\sum_{i<j}\ftilde_{ij}(w_{0})+\frac{n-1}{2}.
\end{align*}
\end{proof}

\begin{proof}[Proof of Proposition \ref{prop:subdominant-efns}]
 According to Theorems \ref{thm:flipriffle-evector}, \ref{thm:rotriffleodd-evector}
and \ref{thm:rotriffleeven-evector}, for $\frac{1}{a^{n}}\oriftauplus$,
$\frac{1}{a^{n}}\oriftauminus$, $\frac{1}{a^{n}}\oriftautildeplus$
and $\frac{1}{a^{n}}\oriftautildeminus$:

\begin{itemize}
\item When $a$ is even, the only subdominant eigenvalue is $1/a$, and
it corresponds to $k=n-1$ and $\bark=0$, i.e. the corresponding
eigenvectors are formed from $n-1$ $\tau$-invariant (or $\tautilde$-invariant)
primitives, and no $\tau$-negating (or $\tautilde$-negating) primitives.
\item When $a$ is odd, the subdominant eigenvalues all correspond to $k=n-1$
and $\bark$ can take any value, i.e. the corresponding eigenvectors
are formed from $n-1$ $\tau$-invariant (or $\tautilde$-invariant)
primitives, and any number of $\tau$-negating (or $\tautilde$-negating)
primitives. However, since $k+\bark\leq n$ (as each primitive has
degree at least one, and the degree of their product is $n$), the
only possible values of $\bark$ are 0 and 1. When $\bark=0$, the
eigenvalue is $1/a$ for all four operator families. When $\bark=1$,
the eigenvalue is $1/a$ for $\frac{1}{a^{n}}\oriftauplus$ and $\frac{1}{a^{n}}\oriftautildeplus$,
and is $-1/a$ for $\frac{1}{a^{n}}\oriftauminus$, and $\frac{1}{a^{n}}\oriftautildeminus$.
\end{itemize}
The proof will be completed by showing the following:
\begin{enumerate}
\item $\{\ftilde_{ij}\ |\:i<j\}\cup\{-\ftilde_{ij}\ |\:j<i\}$ are precisely
all the eigenvectors, for $\oriftautildeplus$ and $\oriftautildeminus$,
given by Algorithm \ref{algo:shuffleevector} when $k=n-1$ and $\bark=0$.
\item $\{\f_{ij}^{+}\ |\:i<j\}$ (resp. $\{\f_{ij}^{-}\ |\:i<j\}$) are,
up to scaling, the sums (resp. differences) of pairs of eigenvectors,
for $\oriftauplus$ and $\oriftauminus$, given by Algorithm \ref{algo:shuffleevector}
when $k=n-1$ and $\bark=0$.
\item $\{\g_{i}\}$ are precisely all the eigenvectors, for $\oriftautildeplus$,
$\oriftautildeminus$, $\oriftauplus$ and $\oriftauminus$ when $a$
is odd, given by Algorithm \ref{algo:shuffleevector} when $k=n-1$
and $\bark=1$.
\end{enumerate}
And here are the proofs of these three assertions.
\begin{enumerate}
\item $k=n-1$ and $\bark=0$ means that the indexing word has $n-1$ Lyndon
factors, whose standard-bracketings are all $\tautilde$-invariant.
So, by the third line of Corollary \ref{cor:stdbrac-invariantornegating},
each Lyndon factor consists of an odd number of positive letters.
Necessarily, $n-2$ of these factors must be single positive letters,
and the remaining factor consists of one positive and one negative
letter - i.e. it is $i\bar{j}$ or $\bari j$, where $i<j$. Take
the first case; the standard-bracketing of $i\bar{j}$ is 
\begin{equation}
[i+\bari,j-\barj]=ij-i\barj+\bari j-\bari\barj-ji-j\bari+\barj i+\barj i.\label{eq:ftilde-stdbrac}
\end{equation}
Observe that, if $w$ consists precisely of $i$ or $\bari$ once
and $j$ or $\barj$ once, then $\ftilde_{ij}(w)$ is precisely the
coefficient of $w$ in (\ref{eq:ftilde-stdbrac}).

Since all primitives involved are $\tautilde$-invariant, the formulas
in parts a, b, c of Theorem \ref{thm:flipriffle-evector} all agree:
the common eigenvector for $\oriftautildeplus$ and $\oriftautildeminus$
is the symmetrised product of (\ref{eq:ftilde-stdbrac}) with $r+\bar{r}$,
ranging over all $r\neq i,j$. In this product, the terms in (\ref{eq:ftilde-stdbrac})
stay as consecutive subwords, and the letters distinct from $i,j$
may have any order, with equal coefficient. Thus coefficient of a
word in this eigenvector is given by $\ftilde_{ij}$.

As for $\bari j$, its standard-bracketing is $[i-\bari,j+\barj]=-[j+\barj,i-\bari]$,
which differs from (\ref{eq:ftilde-stdbrac}) only by a global sign
and the exchange of $i$ and $j$. So, by the same argument as above,
$-\ftilde_{ji}(w)$ describes the coefficient of $w$ in this eigenvector.

\item Similar to case i) above, these eigenvectors are produced from $n-1$
Lyndon factors whose standard-bracketing are all $\tau$-invariant.
By the first line of Corollary \ref{cor:stdbrac-invariantornegating},
each Lyndon factor consists of an even number of negative letters.
As in case i), $n-2$ of the factors are single positive letters,
and the remaining factor is $ij$ or $\bari\barj$. Note that 
\begin{align*}
\stdbrac(ij) & =ij+\bari j+i\barj+\bari\barj-ji-j\bari-\barj i-\barj\bari;\\
\stdbrac(\bari\barj) & =ij-\bari j-i\barj+\bari\barj-ji+j\bari+\barj i-\barj\bari.
\end{align*}
So $p^{+}:=\frac{1}{2}(\stdbrac(ij)+\stdbrac(\bari\barj))$ and $p^{-}:=\frac{1}{2}(\stdbrac(ij)-\stdbrac(\bari\barj))$
have coefficients as decribed by $\f_{ij}^{+}$ and $\f_{ij}^{-}$,
respectively, when $w$ consists only of $i$ or $\bari$ once and
$j$ or $\barj$ once. The extension to words longer than $w$ is
as in case i): take the symmetrised product of $p^{+}$ or $p^{-}$
with all $r+\bar{r}$ for $r\neq i,j$, following Theorems \ref{thm:rotriffleodd-evector}
and \ref{thm:rotriffleeven-evector}. Note that (with the obvious
notation)
\[
\sym\{p^{+}\}\cup\{r+\bar{r}|\ r\neq i,j\}=\sym\{\stdbrac(ij)\}\cup\{r+\bar{r}|\ r\neq i,j\}+\sym\{\stdbrac(\bari\barj)\}\cup\{r+\bar{r}|\ r\neq i,j\}
\]
and similarly for $p^{-}$, because taking symmetrised product is
linear in each argument. 
\item $k=n-1$ and $\bark=1$ means that the indexing word has $n$ Lyndon
factors, so each factor must be a single letter. From Corollary \ref{cor:stdbrac-invariantornegating},
the standard-bracketing of a single positive letter is both $\tau$-invariant
and $\tautilde$-invariant, and the standard-bracketing of a single
negative letter is both $\tau$-negating and $\tautilde$-negating.
Hence, these $n$ Lyndon factors are precisely one negative letter,
say $\bari$, and $n-1$ positive letters. Since $\bark=1$, the formula
in Theorem \ref{thm:flipriffle-evector}.i.c (for $\oriftautildeminus$)
and (\ref{eq:rotriffle-evector}) (for $\oriftauplus$ and $\oriftauminus$)
both simplify to $(i-\bari)s+s(i-\bari)$, where $s$ denotes the
symmetrised product of $r+\bar{r}$ over all $r\neq i$. The coefficient
of a word in this eigenvector is given by $\g_{i}$.
\end{enumerate}
\end{proof}

 \printbibliography[heading=bibintoc]
\end{document}